\documentclass[a4paper,11pt]{amsart}
\usepackage{amsmath,amsthm,amssymb,amsfonts,enumerate,color, upgreek,amsrefs,float}
\usepackage[pdftex]{graphicx}
\usepackage{tikz}
\usetikzlibrary{hobby}
\usepackage[colorlinks=true, allcolors=blue]{hyperref}

\newcommand{\R}{\mathbb{R}}

\newcommand{\vep}{\varepsilon}
\newcommand{\supp}{\operatorname{supp}}

\newcommand{\dist}{\operatorname{dist}}

\newtheorem{thm}{Theorem}[section]
\newtheorem{prop}[thm]{Proposition}
\newtheorem{cor}[thm]{Corollary}
\newtheorem{lem}[thm]{Lemma}
\theoremstyle{definition}
\newtheorem{defn}[thm]{Definition}
\newtheorem{rem}[thm]{Remark}

\numberwithin{equation}{section}

\allowdisplaybreaks

\author[L.~A.~Caffarelli]{Luis A.~Caffarelli}
 
\author[M.~Soria-Carro]{Mar\'ia Soria-Carro}

    \address{Department of Mathematics\\
    The University of Texas at Austin\\
    Austin, TX 78712, USA}
    \email{caffarel@math.utexas.edu, maria.soriac@math.utexas.edu}
    
\author[P.~R.~Stinga]{Pablo Ra\'ul Stinga}

    \address{Department of Mathematics\\
    Iowa State University\\
    396 Carver Hall, Ames, IA 50011, USA}
    \email{stinga@iastate.edu}



\thanks{Research supported by NSF grant 1500871 (Caffarelli) and
Simons Foundation grant 580911 (Stinga)}

\begin{document}

\title[Regularity for $C^{1,\alpha}$ interface transmission problems]{Regularity for $C^{1,\alpha}$ interface transmission problems}

\begin{abstract}
We study existence, uniqueness, and optimal regularity of solutions to transmission problems for harmonic functions
with $C^{1,\alpha}$ interfaces.
For this, we develop a novel geometric stability argument based on the mean value property.
\end{abstract}

\maketitle

\section{Introduction}

Transmission problems in classical elasticity theory were first introduced by M.~Picone in 1954, see \cite{Picone}.
In the following years, contributions were made by J.~L.~Lions \cite{Lions}, G.~Stampacchia \cite{Stampacchia}
and S.~Campanato \cite{Campanato}.
In 1960, M.~Schechter generalized the theory to include smooth elliptic operators in nondivergence form in domains 
with smooth interfaces \cite{Schechter}.
Since then, transmission problems have been of great interest due to their applications in different areas in science.
For instance, O.~A.~Ladyzhenskaya and N.~N.~Ural'tseva considered in \cite{Ladyzhenskaya-Uraltseva}
the so-called diffraction problem.
For other recent developments, see \cite{Citti-Ferrari,Li-Nirenberg,Li-Vogelius,Mateu-Orobitg-Verdera}.

As a particular feature, and in contrast with free boundary problems,
transmission problems deal with a fixed interface where solutions change abruptly and 
the primary focus is to study their behavior across this surface. Additionally, these problems cannot be treated
separately as boundary value problems per se, as solutions interact with each other from each side of the interface
through the transmission condition.

We study existence, uniqueness and regularity of solutions to a transmission problem for harmonic functions.
One of our main novelties is that the transmission interface has only $C^{1,\alpha}$ regularity.
Furthermore, we build up a new fine geometric argument based on the mean value property to
show that solutions are $C^{1,\alpha}$ up to each side of the interface.

The setting is the following. Let $\Omega$ be a smooth, bounded domain of $\R^n$, $n\geq2$. Let $\Omega_1$ be a subdomain of $\Omega$
such that $\Omega_1\subset\subset\Omega$ and set $\Omega_2=\Omega\setminus\overline{\Omega}_1$.
Suppose that the interface $\Gamma$ between $\Omega_1$ and $\Omega_2$, namely,
 $\Gamma=\partial\Omega_1$, is a $C^{1,\alpha}$ manifold, for some $0<\alpha<1$.
Then $\Omega=\Omega_1\cup\Omega_2\cup\Gamma$. 
For a function $u:\overline{\Omega}\to\R$ we denote
$$u_1=u\big|_{\overline{\Omega}_1}\qquad\hbox{and}\qquad u_2=u\big|_{\overline{\Omega}_2}.$$
We consider the problem of finding a continuous function $u:\overline{\Omega}\to\R$ such that
\begin{equation} \tag{TP} \label{eq:transmissionproblem} 
\begin{cases}
\Delta u_1=0&\hbox{in}~\Omega_1\\
\Delta u_2=0&\hbox{in}~\Omega_2\\
u_2=0&\hbox{on}~\partial\Omega\\
u_1=u_2&\hbox{on}~\Gamma\\
(u_1)_\nu-(u_2)_\nu=g&\hbox{on}~\Gamma.
\end{cases}
\end{equation}
Here $g\in C^{0,\alpha}(\Gamma)$ and $\nu$ is the unit
normal vector on $\Gamma$ that is interior to $\Omega_1$, see Figure \ref{figure1}.
This is a transmission problem in the spirit of Schechter in \cite{Schechter}, where $\Gamma$ is the transmission interface.
In contrast to our problem, \cite{Schechter} only deals with $\Gamma\in C^\infty$.
The last two equations on \eqref{eq:transmissionproblem} are called the {\it transmission conditions}. 

\begin{figure}[h]\label{figure1}
\centering
 \begin{tikzpicture}[scale=0.45, use Hobby shortcut, closed=true]
\draw[blue] (-3,0) .. (-3.5,1) ..(-2.5,1)  ..(-2,1.5).. (-1,3.5).. (1.5,3).. (4,3.5).. (4,3)..(4.5,2.5).. (5,2.5).. (5,0.5) ..(4.5,-1).. (3.5,-1.5)..(2.5,-2)..(1,-1).. (0,-2)..(-0.5,-1.5).. (-3,-2).. (-3,0);
\draw (1,0.5) ellipse (7cm and 5.5cm);
\node[] at (1,1) {\large $\Omega_1$};
\node[] at (1,-3.3) {\large $\Omega_2$};
\node[blue] at (5.6,-0.5) {\large $\Gamma$};
\node[] at (8,4) {\large $\Omega$};
\draw[line width=0.5pt, -stealth](-3,-2)--(-2.1,-1) node[anchor=south west]{$\nu$};
    \end{tikzpicture}
\caption{Geometry for the transmission problem \eqref{eq:transmissionproblem}.}
\end{figure}
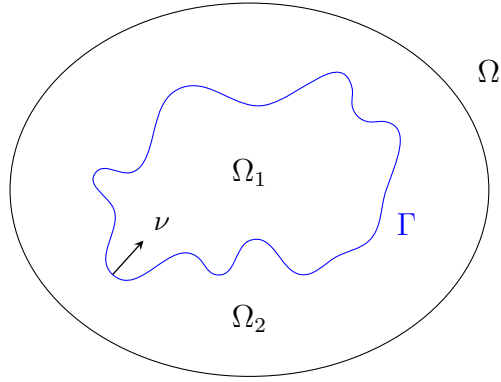

If in \eqref{eq:transmissionproblem} we set $g\equiv0$ then $u$ is a harmonic function in $\Omega$.
Therefore, in order to have a meaningful elliptic transmission condition, we assume that
$$g(x)\geq0\qquad\hbox{for all}~x\in\Gamma.$$
Hence, $u$ will not be differentiable at those points on $\Gamma$ where $g>0$. 
In turn, we prove that $u$ is $C^{1,\alpha}$ from each side up to $\Gamma$.
In \eqref{eq:transmissionproblem} we have also imposed homogeneous Dirichlet boundary condition on $\partial\Omega$.
This is not a restriction since we can always add to $u$ a harmonic function $v$ in $\Omega$ such that
$v=\phi$ on $\partial\Omega$, to make $u_2=\phi$ on $\partial\Omega$. 
The one dimensional case is excluded because one can easily find explicit solutions.

Our main result is the following.

\begin{thm}\label{thm:main}
There exists a unique classical solution $u$ to the transmission problem \eqref{eq:transmissionproblem}.
Moreover, $u_1 \in C^{1,\alpha}(\overline{\Omega}_1)$, $u_2 \in C^{1,\alpha}(\overline{\Omega}_2)$,
and there exists $C=C(n,\alpha,\Gamma)>0$ such that
$$\|u_1\|_{C^{1,\alpha}(\overline{\Omega}_1)}+\|u_2\|_{C^{1,\alpha}(\overline{\Omega}_2)}\leq C\|g\|_{C^{0,\alpha}(\Gamma)}.$$
\end{thm}

The appropriate notion of solution to \eqref{eq:transmissionproblem} comes from computing $\Delta u$
in the sense of distributions. Indeed, if $u$ and $\Gamma$ were sufficiently smooth and
$\varphi\in C^\infty_c(\Omega)$ then
\begin{align*}
(\Delta u)(\varphi) &= \int_\Omega u\Delta\varphi\,dx=\int_\Gamma\big((u_1)_\nu-(u_2)_\nu\big)\varphi\,dH^{n-1}=\int_\Gamma g\varphi\,dH^{n-1}.
\end{align*}
Thus $\Delta u$ is a singular measure concentrated on $\Gamma$ with density $g$. 
In Section \ref{Section:existence} we show that there exists a unique distributional solution $u\in C_0(\overline{\Omega})$ to \eqref{eq:transmissionproblem},
where $C_0(\overline{\Omega})$ denotes the space of continuous functions on $\overline{\Omega}$
that vanish on $\partial\Omega$. Moreover, we prove 
that $u$ is Log-Lipschitz in $\overline{\Omega}$, see Theorem \ref{thm:existence}. 
The main issue is the optimal regularity of $u$ up to $\Gamma$.
Theorem \ref{thm:main} will be a consequence of our next result.

\begin{thm}[Pointwise $C^{1,\alpha}$ boundary regularity] \label{thm:boundary}
Let $\Gamma=\{(y',\psi(y')):y'\in B_1'\}$, where $\psi$ is a $C^{1,\alpha}$ function, for some $0<\alpha<1$. Assume that $0\in \Gamma$. Let $u\in C(\overline{B_1})$ be a distributional solution to the transmission problem
$$\Delta u=g\,dH^{n-1}\big|_{\Gamma}$$
where $g\in L^\infty(\Gamma)$, {$g\geq 0$,} and $g\in C^{0,\alpha}(0)$. 
Then there are linear polynomials $P(x) = A \cdot x+B$, and $Q(x)=C\cdot x + B$ such that
\begin{align*}
|u_1(x)-P(x)|&\leq D |x|^{1+\alpha} \qquad \hbox{for all}~x\in \Omega_1 \cap B_{1/2}\\
|u_2(x)-Q(x)|&\leq D |x|^{1+\alpha} \qquad \hbox{for all}~x\in \Omega_2 \cap B_{1/2}
\end{align*}
with
$$
|A|+|B|+|C| + D\leq C_0 \|\psi\|_{C^{1,\alpha}(B_1')}\big([g]_{C^{\alpha}(0)}+\Vert g \Vert_{L^\infty{(\Gamma)}}\big)
$$
and $C_0=C_0(n,\alpha)>0$.
\end{thm}

The key tool to prove Theorem~\ref{thm:boundary} is a \emph{stability} result,
obtained via the novel geometric approach we develop, which is
based on the mean value property and the maximum principle, see Theorem~\ref{thm:stability}. 
In fact, our idea is to explicitly construct classical solutions to problems with flat interfaces that are close to $u$.
With this, we can transfer the regularity from classical solutions to $u$. Indeed, as shown in Section \ref{Section:flat},
solutions to flat problems have the expected optimal regularity up to the interface. 
More precisely, we show that if the flatness and oscillation of the interface $\Gamma$ are controlled,
then we can construct a solution for a flat interface problem, where the flat interface does not intersect $\Gamma$.
We also quantify how close solutions must be, depending only on the geometric properties of $\Gamma$
and the basic regularity of $u$. These ingredients are crucial for the first step in the proof of
Theorem~\ref{thm:boundary}, see Lemma~\ref{lem:basic}.
To close the argument, one needs to use these approximations at each scale.
Through this techniques, and similar to the case of elliptic equations \cite{Caffarelli},
we are able to find that flat solutions are asymptotically close to non-flat solutions.

Our geometric techniques developed in Section~\ref{Section:stability} are constructive and quantitative,
and provide a precise understanding of the underlying geometry of the transmission problem.
Furthermore, this paper is essentially self-contained.
We believe that the tools presented here could be used in free boundary problems,
an idea we will explore in the future.
Finally, notice that our results are also useful in terms of numerical analysis,
as our constructions give explicit rates of approximation.

The paper is organized as follows. In Section \ref{Section:existence} we prove existence, uniqueness and
global Log-Lipschitz regularity of the solution $u$ to \eqref{eq:transmissionproblem}.
Section \ref{Section:flat} deals with the case when the transmission interface is flat.
Our geometric stability result based on the mean value property is proved in Section 4. The proof of Theorems~\ref{thm:boundary} and \ref{thm:main} are given in Sections \ref{Section:boundary} and \ref{Section:proofofmain}, respectively. 
The appendix contains some basic geometric considerations about integration on Lipschitz domains.

\medskip

\noindent\textbf{Notation.} For a point $x\in\R^n$ we write $x=(x',x_n)$, where $x'\in\R^{n-1}$, $x_n\in\R$.
The gradient in the variables $x'$ is denoted by $\nabla'$, $dH^{n-1}$ is the $(n-1)$-dimensional Hausdorff measure in $\R^n$
and $B_r'(x')$ denotes the ball in $\R^{n-1}$
of radius $r>0$ centered at $x'$. When the ball is centered at the origin $x'=0'$ or $x=0=(0',0)$, we will just write $B_r'$ or $B_r$.

\section{Existence, uniqueness and global Log-Lipschitz regularity}\label{Section:existence}

As we mentioned in the Introduction, the notion of solution to \eqref{eq:transmissionproblem} comes from computing $\Delta u$
in the sense of distributions.

\begin{defn}[Distributional solution]
We say that $u\in C_0(\overline\Omega)$ is a distributional solution to \eqref{eq:transmissionproblem} if for any $\varphi\in C^\infty_c(\Omega)$ we have
$$\int_\Omega u\Delta\varphi\,dx=\int_\Gamma g\varphi\,dH^{n-1}.$$
In this case, we write
$$\Delta u=g\,dH^{n-1}\big|_{\Gamma}.$$
\end{defn}

Even though the definition of distributional solution makes sense for $u\in L^1_{\rm loc}(\Omega)$,
we ask $u$ to be continuous up to the boundary so that the boundary condition $u=0$ is well-defined.

Recall that a bounded function $u:\overline{\Omega}\to\R$ is in the space ${\rm LogLip}(\overline\Omega)$ if
$$[u]_{{\rm LogLip}(\overline\Omega)}=\sup_{\substack{x,y\in\Omega\\x\neq y}}\frac{|u(x)-u(y)|}{|x-y||\log|x-y||}<\infty.$$

\begin{thm}[Existence, uniqueness, and Log-Lipschitz global regularity]\label{thm:existence}
Let $\Gamma$ be a Lipschitz interface, and $g\in L^\infty(\Gamma)$.
Then the unique distributional solution $u\in C_0(\overline\Omega)$ to \eqref{eq:transmissionproblem} is given by
\begin{equation}\label{eq:uwithG}
u(x)=\int_\Gamma G(x,y)g(y)\,dH^{n-1} \qquad\hbox{for}~x\in \Omega
\end{equation}
where $G(x,y)$ is the Green's function for the Laplacian in $\Omega$. 
Furthermore, $u\in {\rm LogLip}(\overline\Omega)$ and there exists $C=C(n,\Gamma,\Omega)>0$ such that
$$\|u\|_{L^\infty(\Omega)}\ + [u]_{{\rm LogLip}(\overline{\Omega})}\leq C\|g\|_{L^\infty(\Gamma)}.$$
\end{thm}

\begin{proof}
Let $u$ be as in \eqref{eq:uwithG}.
By using a partition of unity on $\Gamma$, it is enough to assume
that $\Gamma=\psi(\R^{n-1})$ where $\psi:\R^{n-1}\to\R$ is a Lipschitz function
and that $g(y',\psi(y'))$ has compact support in $B_1'$ (see Appendix~\ref{app:special}). Then, for any $x\in\Omega$,
\begin{align*}
|u(x)| &\leq \int_\Gamma|G(x,y)|g(y)\,dH^{n-1}_y \\
&=\int_{B_1'}|G(x,(y',\psi(y')))|g(y',\psi(y'))\sqrt{1+|\nabla'\psi(y')|^2}\,dy' \\
&\leq C(n,\Gamma)\|g\|_{L^\infty(\Gamma)}\int_{B_1'}\frac{1}{|(x'-y',x_n-\psi(y'))|^{n-2}}\,dy'\\
&\leq C(n,\Gamma)\|g\|_{L^\infty(\Gamma)}\int_{B_1'}\frac{1}{|x'-y'|^{n-2}}\,dy'\\
&\leq C(n,\Gamma)\|g\|_{L^\infty(\Gamma)}.
\end{align*}
Thus the integral defining $u$ in \eqref{eq:uwithG} is absolutely convergent and $u$ is bounded.

Next, for any $\varphi\in C^\infty_c(\Omega)$, by Fubini's Theorem and the symmetry $G(x,y)=G(y,x)$,
\begin{align*}
\int_\Omega u(x)\Delta\varphi(x)\,dx &= \int_\Omega\bigg[\int_\Gamma G(x,y)g(y)\,dH^{n-1}\bigg]\Delta\varphi(x)\,dx \\
&=\int_\Gamma g(y)\int_\Omega G(y,x)\Delta_x\varphi(x)\,dx\,dH^{n-1}= \int_\Gamma g(y)\varphi(y)\,dH^{n-1}.
\end{align*}
Moreover, since $G(\bar{x},y)=0$ for $\bar{x}\in\partial\Omega$ and $y\in\Omega$, by dominated convergence we see that $u(x)$ converges to $0$
as $x\in\Omega$ converges to $\bar{x}$.

Now we show that $u\in\mathrm{LogLip}(\overline{\Omega})$. Since $u$ is harmonic in $\Omega\setminus \Gamma$,
we only need to prove the regularity of $u$ near $\Gamma$. Suppose that $x_1,x_2\in K$, where $K\subset\Omega$ is a compact set containing $\Gamma$. Let $0<d<<1$. If $|x_1-x_2|\geq d$ then
$$|u(x_1)-u(x_2)|\leq \frac{2\|u\|_{L^\infty(\Omega)}}{d}d\leq C|x_1-x_2|.$$
Assume next that $|x_1-x_2|=\delta<d$.
If $n\geq 3$ then, since $B_{2\delta}(x_1)\subset B_{4\delta}(x_2)$, by classical estimates
for the Green's function,
\begin{align*}
|u(x_1)&-u(x_2)| \leq \int_{\Gamma}|G(x_1,y)-G(x_2,y)||g(y)|\,dH^{n-1} \\
&\leq C_{n,K}\|g\|_{L^\infty(\Gamma)}\bigg[\int_{B_{2\delta}(x_1)\cap\Gamma}\frac{1}{|x_1-y|^{n-2}}\,dH^{n-1}+\int_{B_{4\delta}(x_2)\cap\Gamma}\frac{1}{|x_2-y|^{n-2}}\,dH^{n-1} \\
&\quad\qquad\qquad\qquad\quad+\int_{\Gamma\setminus(B_{2\delta}(x_1)\cap\Gamma)}\frac{|x_1-x_2|}{|x_1-y|^{n-1}}\,dH^{n-1}\bigg] \\
&\leq C_{n,K,\Gamma}\|g\|_{L^\infty(\Gamma)}\bigg[\int_{B_{2\delta}'(x_1')}\frac{1}{|x_1'-y'|^{n-2}}\,dy'+\int_{B_{4\delta}'(x_2')}\frac{1}{|x_2'-y'|^{n-2}}\,dy' \\
&\qquad\qquad\qquad\qquad\quad +|x_1-x_2|\int_{B_1'\setminus B_{2\delta}'(x_1')}\frac{1}{|x_1'-y'|^{n-1}}\,dy'\bigg] \\
&\leq C_{n,K,\Gamma}\|g\|_{L^\infty(\Gamma)}\big(|x_1-x_2|+|x_1-x_2||\log|x_1-x_2||\big).
\end{align*}
The estimate in dimension $n=2$ follows the same lines.

For uniqueness, if $u,v\in C_0(\overline{\Omega})$ are distributional solutions then
$$\int_\Omega (u-v)\Delta\varphi\,dx=0 \quad \hbox{for every}~\varphi\in C^\infty_c(\Omega).$$
Hence, $u-v\in C_0(\overline{\Omega})$ is harmonic in $\Omega$ and, as a consequence, $u\equiv v$.
\end{proof}

\begin{rem}\label{rem:gamma}
Note that if $u\in  {\rm LogLip}(\overline\Omega)$ then $u\in C^{0,\gamma}(\overline \Omega)$ for every $0<\gamma <1$
and there exists $C=C(\Omega,\gamma)>0$ such that
$$
[u]_{C^{0,\gamma}(\overline \Omega)} \leq C [u]_{{\rm LogLip}(\overline{\Omega})}.
$$
\end{rem}

\section{Flat problems}\label{Section:flat}

For the next results, we fix the following notation. For $a\in\R$ we denote
\begin{align*}
B_{r,a} & =B_r(0',a)\\
B_{r,a}^+ & = B_r(0',a) \cap \{x_n >a\}\\
B_{r,a}^- & = B_r(0',a) \cap \{x_n <a\}\\
T_{r,a} &= \{x\in B_r(0',a):x_n=a\}\\
T_a &=B_1 \cap \{x_n=a\}\\
 T_a^+&=\{x_n\geq a\}\\ 
 T_a^-&=\{x_n\leq a\}.
\end{align*}
When $a=0$, we use the simplified notation $T=T_{0}$ and $B_r^{\pm}=B_{r,0}^{\pm}$.

\begin{thm}[Flat problem]\label{thm:flat}
Let $r>0$ and $a\in \R$. Given $0<\alpha, \gamma<1$, let $g\in C^{0,\alpha}(T_{r,a})$ and $f\in C^{0,\gamma}(\partial B_{r,a})$. Then there exists a unique
solution $v\in C^\infty(B_{r,a}\setminus T_{r,a})\cap C^{0,\gamma}(\overline{B_{r,a}})$ to the flat transmission problem
$$\begin{cases}
\Delta v = g\, dH^{n-1}\big|_{T_{r,a}}&\hbox{in}~B_{r,a}\\
v=f&\hbox{on}~\partial B_{r,a}
\end{cases}$$
that satisfies the global estimate
$$
\|v\|_{C^{0,\gamma}(\overline{B_{r,a}})} \leq C\big(\|g\|_{C^{0,\alpha}(T_{r,a})} + \| f\|_{ C^{0,\gamma}(\partial B_{r,a})}\big)
$$
where $C= C(n,\alpha,\gamma,r)>0$.
Moreover, if we let $v^\pm = v \chi_{ \overline{B_{r,a}^\pm}}$, then $v^\pm\in C^{1,\alpha}(\overline{B_{r/2,a}^\pm})$ and
$$
\| v^\pm \|_{C^{1,\alpha}(\overline{B_{r/2,a}^\pm})} \leq C\big(\|g\|_{C^{0,\alpha}(T_{r,a})} + \| f\|_{ L^\infty({\partial B_{r,a}})}\big)
$$
where $C= C(n,\alpha,r)>0$.
If $g\in C^{k-1,\alpha}(T_{r,a})$, $k\geq 1$, then $v\in C^{k,\alpha}(\overline{B_{r/2,a}^\pm})$ and 
$$
 \| v^\pm \|_{C^{k,\alpha}(\overline{B_{r/2,a}^\pm})} \leq C \big(\|g\|_{C^{k-1,\alpha}(T_{r,a})}+  \| f\|_{L^\infty(\partial B_{r,a})} \big)
$$
where $C=C(n,\alpha,r,k)>0$.
\end{thm}

\begin{proof}
By subtracting from $v$ the harmonic function $h$ in $B_{r,a}$ that coincides with $f$ on $\partial B_{r,a}$,
it is enough to assume that $f=0$ on $\partial B_{r,a}$.
We consider only the case $k=1$, that is, $g\in C^{0,\alpha}(T_{r,a})$. When $k\geq1$ the proof is completely analogous.
Moreover, it is sufficient to prove the result for $a=0$ and $r=1$. Indeed suppose that $g$ is as in the statement, and let $\tilde{g}$ be defined on $T$, so that
$$
{g}(x',x_n)=r^{-1} \tilde g\big(r^{-1}x',r^{-1}(x_n-a)\big)
$$
whenever $x\in{T_{r,a}}$. If $\tilde{v}$ is the corresponding solution in $B_1$, then 
$$
{v} (x',x_n) = \tilde v\big(r^{-1} x', r^{-1} (x_n -a)\big)\qquad\hbox{for}~x\in\overline{B_{r,a}}
$$
is the unique solution to $\Delta v = g\, dH^{n-1}\big|_{T_{r,a}}$ such that $v=0$ on $\partial B_{r,a}$.
Moreover, we have the following control of the norms:
\begin{align*}
\| v^\pm \|_{C^{1,\alpha}(\overline{B_{r/2,a}^\pm})}
&= \|\tilde v^\pm \|_{L^\infty(\overline{B_{1/2}^\pm})} +r^{-1} \| \nabla \tilde v^\pm \|_{L^\infty(\overline{B_{1/2}^\pm})} + r^{-(1+\alpha)} [\nabla \tilde v^\pm]_{C^{0,\alpha}(\overline{B_{1/2}^\pm})} \\
&\leq  \max\{1,r^{-1},r^{-(1+\alpha)}\} \| \tilde v^\pm \|_{C^{1,\alpha}(\overline{B_{1/2}^\pm})} \\
&\leq  C \max\{1,r^{-1},r^{-(1+\alpha)}\} \|\tilde g\|_{C^{0,\alpha}(T)} \\
& \leq C  \max\{1,r^{-1},r^{-(1+\alpha)}\} \big(r \| g\|_{L^\infty(T_{r,a})} + r^{1+\alpha} [ g]_{C^{0,\alpha}(T_{r,a})}\big)\\
& \leq C \|g\|_{C^{0,\alpha}(T_{r,a})},
\end{align*}
and, similarly,
$$
\|v\|_{C^{0,\gamma}(\overline{B_{r,a}})} \leq C \|g \|_{C^{0,\alpha}(T_{r,a})},
$$
where $C>0$ is as in the statement.

Let $v^{+}$ be the solution to the mixed boundary value problem
$$\begin{cases}
\Delta  v^+ =  0 & \hbox{in}~B_1^+\\
 v^+  =  0 & \hbox{on}~\partial B_1^+\setminus {T}\\
 v^+_{x_n}  =  g/2 &\hbox{on}~T.
\end{cases}$$ 
By classical elliptic regularity, $ v^+\in C^\infty(B_1^+)\cap C^{1,\alpha}(\overline{B_{1/2}^+})$ and
$$\|  v^+ \|_{C^{1,\alpha}(\overline{B_{1/2}^+})} \leq C_0 \|  g\|_{C^{0,\alpha}(T)}$$
for some $C_0=C_0(n)>0$.
Furthermore, $v^+\in C^{0,\gamma}(\overline{B_1^+})$.
Indeed, consider the solution $w$ to
$$
\begin{cases}
\Delta w = 0 & \hbox{in}~B_2^+\\
w =0 & \hbox{on}~\partial B_2^+ \setminus T_{2,0}\\
w_{x_n} = \tilde g/2 & \hbox{on}~T_{2,0}
\end{cases}
$$
where $\tilde g=g$ on $T$ with $\|\tilde g\|_{C^{0,\alpha}(T_{2,0})} \leq \tilde{C} \|g\|_{C^{0,\alpha}(T)}$,
for some constant $\tilde{C}>0$. Then $w\in C^\infty(B_2^+)\cap C^{1,\alpha}(\overline{B_1^+})$ with
$$
\| w\|_{C^{1,\alpha}(\overline{B_1^+})} \leq C_1 \|\tilde g\|_{C^{0,\alpha}(T_{2,0})} \leq C_1\tilde{C}\|g \|_{C^{0,\alpha}(T)}
$$
where $C_1=C_1(n)$. Define $u(x)=v^+(x)-w(x)$, for $x\in\overline{B_1}$, and  
consider the even reflection extension of $u$ to $\overline{B_1}$ given by $\tilde{u}(x',x_n)=u(x',|x_n|)$.
It follows that $\tilde u$ is harmonic in $B_1$ and $\tilde u = -\tilde w$ on $\partial B_1$, where $\tilde w$ is the even reflection of $w$
to $\overline{B_1}$. 
Since $\tilde w \in\mathrm{Lip}(\overline{B_1})$, by using the Poisson kernel in $B_1$ (see \cite{Gilbarg-Trudinger}),
it can be checked that
$$
\| \tilde u \|_{C^{0,\gamma}(\overline{B_1})} \leq C \| \tilde w\|_{C^{0,\gamma}(\partial B_1)}\leq C \|g\|_{C^{0,\alpha}(T)}
$$
where $C=C(n,\alpha,\gamma)>0$.
Therefore, $v^+\in C^{0,\gamma}(\overline{B_1^+})$, with the corresponding estimate
for $\| v^+ \|_{C^{0,\gamma}(\overline{B_1^+})}$ as in the statement.
Next, the function $ v^-(x',x_n) =  v^+(x',-x_n)$ solves
$$\begin{cases}
\Delta  v^-  =  0 & \hbox{in}~B_1^-\\
 v^-  =  0 & \hbox{on}~\partial B_1^-\setminus T\\
 v^-_{x_n} = -  g/2 & \hbox{on}~T,
\end{cases}$$
and $v^- \in C^\infty(B_1^-)\cap C^{1,\alpha}(\overline{B_{1/2}^-})\cap C^{0,\gamma}(\overline{B_1^-})$.
It follows that 
$v= v^+\chi_{\overline{B_1^+}} +  v^- \chi_{\overline{{B_1^-}}}$
 is the unique distributional solution to 
$\Delta  v = g\, dH^{n-1}\big|_{T}$ such that $ v=0$ on $\partial B_1$. Furthermore, $v\in C^\infty(B_1\setminus T)\cap C^{0,\gamma}(\overline{B_1})$ and $ v^\pm \in C^{1,\alpha}(\overline{B_{1/2}^\pm})$ with
$$\|v\|_{C^{0,\gamma}(\overline{B_1})}\leq C(n,\alpha,\gamma) \|  g\|_{C^{0,\alpha}(T)}$$
and
$$\| v^\pm \|_{C^{1,\alpha}(\overline{B_{1/2}^\pm})}  \leq C(n,\alpha) \|  g\|_{C^{0,\alpha}(T)}.$$
\end{proof}

\begin{cor} \label{cor:flat1}
Given $|a|<1/4$, $c_0>0$, and $f\in C^{0,\gamma}(\partial B_1)$, with $0<\gamma<1$, there exists a unique solution $v\in C^\infty(B_1\setminus T_a)\cap C^{0,\gamma}(\overline{B_1})$  to 
$$\begin{cases}
\Delta v = c_0 \, dH^{n-1}|_{T_a}&\hbox{in}~B_1\\
v=f&\hbox{on}~\partial B_1
\end{cases}$$
such that 
$$
\|v\|_{C^{0,\gamma}(\overline{B_1})} \leq C\big(c_0 + \| f\|_{ C^{0,\gamma}(\partial B_1)}\big)\\
$$
where $C=C(n,\gamma)>0$ and, for any $k\geq1$, 
\begin{align*}
 \| v^\pm \|_{C^{k,\alpha}(\overline{B_{1/2}}\cap T_a^\pm)} &\leq C\big(c_0 + \| f\|_{ L^\infty({\partial B_1})}\big)
\end{align*}
where $C=C(n,\alpha,k)>0$. 
\end{cor}

\begin{proof} 
The global $C^{0,\gamma}$ estimate follows immediately from Theorem~\ref{thm:flat} with $g=c_0$. Hence, we only need to show the $C^{k,\alpha}$ estimate.
Fix $k\geq 1$. By Theorem~\ref{thm:flat} with $r=4$, there is a unique solution $w \in C^\infty(B_{4,a}\setminus T_{4,a})\cap C^{0,\gamma}(\overline{B_{4,a}})$ to 
 $
\Delta w = c_0 \, dH^{n-1}\big|_{T_{4,a}}
$
such that $w=0$ on $\partial B_{4,a}$. Moreover, 
$
\| w^\pm \|_{C^{k,\alpha}(\overline{B_{2,a}^\pm})} \leq C c_0,
$
for some $C=C(n,\alpha,k)>0$. Let $h$ be the harmonic function in $B_1$ such that $h=w-f$ on $\partial B_1$. Then $h\in C^\infty(B_1)\cap C^{0,\gamma}(\overline{B_1})$, and 
$$
\|h \|_{C^{k,\alpha}(\overline{B_{1/2}})}\leq C \big( \| w \|_{L^\infty(\partial B_1)}+ \| f \|_{L^\infty(\partial B_1)}\big)\leq C \big(c_0 + \| f \|_{L^\infty(\partial B_1)}\big)
$$
where $C=C(n,\alpha,k)>0$. Define $v=w-h$ on $\overline{B_1}$. Then $v$ is the unique solution to $\Delta v = g \, dH^{n-1}|_{T_a}$ with $v=f$ on $\partial B_1$. Moreover, since $\overline{B_{1/2}}\cap T_a^\pm \subset \overline{B_{2,a}^\pm}$, 
$$
\| v^\pm \|_{C^{k,\alpha}(\overline{B_{1/2}}\cap T_a^\pm)} \leq \| w^\pm \|_{C^{k,\alpha}(\overline{B_{2,a}^\pm})} + \| h \|_{C^{k,\alpha}(\overline{B_{1/2}})} \leq C \big(c_0 + \| f\|_{ L^\infty({\partial B_1})}\big).
$$
\end{proof}

\section{The stability result}\label{Section:stability}

In this section we prove our stability result, Theorem \ref{thm:stability}.
As we mentioned at the beginning, our argument is based on the mean value property
and, therefore, it is self-contained.

Fix $\vep>0$, and let $\Omega_\vep=\{x\in \Omega : d(x,\partial\Omega)>\vep\}$
and $\Gamma_\vep =\{ x\in \Omega : d(x,\Gamma)<\vep\}$.  
Consider the average
$$u_{\vep}(x) =\frac{1}{|B_\vep|} \int_{B_\vep(x)} u(y)\, dy\qquad\hbox{for }~x\in \Omega_\vep.$$

\begin{prop}[Properties of averages] \label{lem:averages}  Let $u$
be the distributional solution given in Theorem \ref{thm:existence}.  The following properties hold.
\begin{enumerate}[$(i)$]
\item If $B_\vep(x) \cap \Gamma = \varnothing$ then  $u_\vep (x) = u(x)$.
\item $u_\vep \to u$ uniformly in compact subsets of $\Omega$, as $\vep \to 0$.
\item If $g\in L^\infty(\Gamma)$ then $g_\vep\in C_c(\Gamma_\vep)$, where
$$
g_\vep(x) = \frac{1}{|B_\vep|}\int_{\Gamma\cap B_\vep(x)}g(y)\,dH^{n-1} \qquad\hbox{for }~x\in \Gamma_\vep.
$$
Moreover, $\Delta u_\vep (x) = g_\vep(x)$ for any $x\in \Omega_\vep$.
\end{enumerate}
\end{prop}

\begin{proof}
Since $u$ is harmonic outside of $\Gamma$, $(i)$ is immediate by the mean value property.

For $(ii)$, recall by Remark \ref{rem:gamma} that $u\in C^{0,\gamma}(\overline\Omega)$. Therefore,
$$
|u_\vep(x)-u(x)| \leq \frac{1}{|B_\vep|} \int_{B_\vep(x)} |u(y)-u(x)|\, dy \leq C\|g\|_{L^\infty(\Gamma)} \vep^\gamma \to 0
$$ 
as $\vep \to 0$.

We now show $(iii)$. If $g\in L^\infty(\Gamma)$, by dominated convergence,
$g_\vep\in C_c(\Gamma_\vep)$. Moreover, for any $\varphi\in C^\infty_c(\Omega)$, we have
\begin{align*}
(\Delta u_\vep)(\varphi) &= \int_{{\Omega}}u_\vep(x)\Delta\varphi(x)\,dx \\
&=\frac{1}{|B_\vep|}\int_{B_\vep}\int_{\Omega}u(x+y)\Delta\varphi(x)\,dx\,dy \\
&=\frac{1}{|B_\vep|}\int_{B_\vep}\int_{\Omega}u(z)\Delta\varphi(z-y)\,dz\,dy \\
&=\frac{1}{|B_\vep|}\int_{B_\vep}\int_{\Gamma}g(z)\varphi(z-y)\,dH^{n-1}_z\,dy \\
&=\frac{1}{|B_\vep|}\int_{\Gamma}\bigg[\int_{B_\vep}\varphi(z-y)\,dy\bigg]\,g(z)\,dH^{n-1}_z \\
&=\frac{1}{|B_\vep|}\int_{\Gamma}\bigg[\int_{\Omega}\chi_{B_{\vep}}(z-y)\varphi(y)\,dy\bigg]\,g(z)\,dH^{n-1}_z \\
&=\frac{1}{|B_\vep|}\int_{\Omega}\int_{\Gamma}\chi_{B_{\vep}}(z-y)\,g(z)\,dH^{n-1}_z\,\varphi(y)\,dy \\
&=\int_{\Omega}\bigg[\frac{1}{|B_\vep|}\int_{\Gamma\cap B_\vep(y)}g(z)\,dH^{n-1}_z\bigg]\,\varphi(y)\,dy=\int_{\Omega}g_\vep(y)\varphi(y)\,dy.
\end{align*}
\end{proof}

\begin{thm}[Stability]\label{thm:stability}
Let $0<\vep,\theta<1/2$ and $0<\delta,\gamma<1$ be given, and 
let $\Gamma=\{(y',\psi(y')):y'\in B_1'\}$, where $\psi$ is a Lipschitz function.
Assume that $\Gamma$ is $\theta\vep$-flat in $B_1$ in the sense that
$$\Gamma\subset\{x\in B_1:|x_n|<\theta\vep\}$$
and that $\Gamma$ is also {$\vep$-horizontal} in $B_1$, that is,
$$
1-\vep \leq \nu(x) \cdot (0',1) = \big(1+|\nabla'\psi(x')|^2\big)^{-1/2} \leq 1
$$
for every $x\in \Gamma$, where $\nu(x)$ denotes the upward pointing normal on $\Gamma$.
Then there exists $C=C(n,\gamma)>0$ such that for any $u\in C(\overline{B_1})$ and $g\in L^\infty(\Gamma)$ satisfying
$$\begin{cases}
\Delta u=g\,dH^{n-1}\big|_{\Gamma}&\hbox{in}~B_1\\
|u|\leq 1&\hbox{in}~B_1\\
|g-1|\leq\delta&\hbox{on}~\Gamma
\end{cases}$$
the classical solution $v\in C^\infty(B_{3/4} \setminus T_{-\theta\vep})\cap C^{0,\gamma}(\overline{B_{3/4}})$ to the flat problem
$$\begin{cases}
\Delta v= dH^{n-1}\big|_{T_{-\theta\vep}}&\hbox{in}~B_{3/4}\\
v=u&\hbox{on}~\partial B_{3/4}
\end{cases}$$
satisfies
\begin{equation}\label{eq:estimate}
|u-v|\leq C(\theta+\delta+\vep^\gamma)\qquad\hbox{in}~B_{1/2}.
\end{equation}
\end{thm}

\begin{rem}\label{rem:above}
The interface for the flat problem in Theorem \ref{thm:stability}
is $T_{-\theta\vep}=B_{3/4}\cap\{x_n=-\theta \vep\}$, which lies below
$\Gamma$ in the $x_n$-direction. To approximate $u$ with the solution to a flat problem where the interface
lies above $\Gamma$ in the $x_n$-direction, it is enough to consider the classical solution $v$ to
$$\begin{cases}
\Delta v= dH^{n-1}\big|_{T_{\theta\vep}}&\hbox{in}~B_{3/4}\\
v=u&\hbox{on}~\partial B_{3/4}.
\end{cases}$$
In this case, the same conclusion as in Theorem \ref{thm:stability} holds. 
\end{rem}

Before proceeding with the proof, we need the following geometric result.

\begin{lem}\label{lem:medidas}
Let $\Gamma$ be as in Theorem \ref{thm:stability}. Define ${M=1+2\theta}$ and let
$x\in B_{3/4-M\vep}$ be such that $\dist(x,\Gamma)<\vep$. Then
\begin{equation}\label{eq:1}
\{y':(y',\psi(y'))\in B_\vep(x)\}\subset B_{((M\vep)^2-(x_n+\theta\vep)^2)^{1/2}}'(x')=\{y':(y',-\theta\vep)\in B_{M\vep}(x)\}
\end{equation}
and
\begin{equation}\label{eq:2}
\{y':(y',\psi(y'))\in B_{M\vep}(x)\}\supset B'_{(\vep^2-(x_n+\theta\vep)^2)^{1/2}}(x')=\{y':(y',-\theta\vep)\in B_\vep(x)\}.
\end{equation}
\end{lem}

\begin{figure}[h]
\centering
\begin{tikzpicture}[scale=2]
\draw[densely dotted] (0,0) circle (1cm);
\draw[densely dotted] (0,0) circle (1.5cm);
\draw[black,fill=black] (0,0) circle (.2ex);
\draw[very thin, densely dotted] (-2,-0.6) -- (2,-0.6);
\draw[very thin, densely dotted] (-2,-0.3) -- (2,-0.3);
\draw[] (-2,-0.9) -- (2,-0.9);
\draw[densely dotted, very thin] (-1.2,-0.9) -- (-1.2,-0.6);
\draw[densely dotted, very thin] (1.2,-0.9) -- (1.2,-0.6);
\draw[densely dotted, very thin] (-0.68,-0.73) -- (-0.68,-0.6);
\draw[densely dotted, very thin] (0.9,-0.45) -- (.9,-0.6);
\draw[thick, red] (-1.2,-0.6) -- (-0.68,-0.6);
\draw[thick, red] (0.9,-0.6) -- (1.2,-0.6);
\draw[blue] plot[smooth,domain=-2:2] (\x, {0.2*sin(deg(\x))-0.6} );
\node[] at (0,-0.15) {\footnotesize $x$};
\node[] at (2.2,-0.6) {\footnotesize $0$};
\node[] at (0,0.6) {\footnotesize $B_\vep(x)$};
\node[] at (1.4,1.3) {\footnotesize $B_{M\vep}(x)$};
\node[blue] at (-2.2,-0.7) {\footnotesize $\Gamma$};
\node[] at (2.25,-0.9) {\footnotesize $-\theta\varepsilon$};
\node[] at (2.2,-0.3) {\footnotesize $\theta\varepsilon$};
\end{tikzpicture}
\caption{The red set is $\{y':(y',-\theta\vep)\in B_{M\vep}(x)\}\setminus \{y':(y',\psi(y'))\in B_\vep(x)\}$.}
\end{figure}
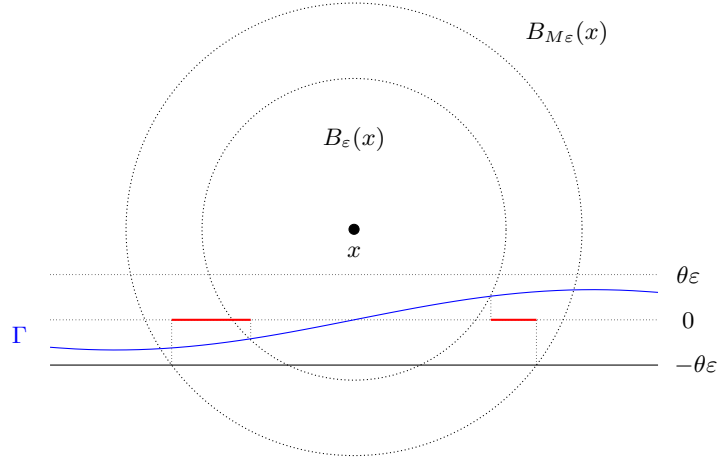

\begin{proof}
If $x$ is as in the statement then, by the flatness condition on $\Gamma$, we have $|x_n|<(1+\theta)\vep$.

Let us prove \eqref{eq:1}. Suppose first that $-\theta\vep<  x_n<\theta\vep$. Then
$$\{y':(y',\psi(y'))\in B_\vep(x)\} \subset \{y':(y',x_n)\in B_\vep(x)\}=B'_\vep(x').$$
Since
\begin{align*}
(M\vep)^2-(x_n+\theta\vep)^2 &=(1+2\theta)^2\vep^2-(x_n+\theta\vep)^2 \\
&\geq (1+4\theta+4\theta^2)\vep^2 - (2\theta\vep)^2=\vep^2+4\theta\vep^2>\vep^2
\end{align*}
we see that $B'_\vep(x')\subset B_{((M\vep)^2-(x_n+\theta\vep)^2)^{1/2}}'(x')$ and the conclusion follows. 
Assume now that $\theta\vep\leq x_n<(1+\theta)\vep$. Notice that
$$\{y':(y',\psi(y'))\in B_\vep(x)\} \subset \{y':(y',\theta\vep)\in B_\vep(x)\}=B'_{(\vep^2-(x_n-\theta\vep)^2)^{1/2}}(x').$$
Since
\begin{align*}
(M\vep)^2-&(x_n+\theta\vep)^2 - (\vep^2-(x_n-\theta\vep)^2)\\
 &= (1+2\theta)^2\vep^2 -(x_n^2 + 2\theta\vep x_n + (\theta\vep)^2)- \vep^2 + (x_n^2 -2\theta\vep x_n + (\theta\vep)^2)\\
&= 4\theta\vep^2 + 4\theta^2\vep^2 - 4\theta\vep x_n \geq 4\theta\vep^2\geq 0
\end{align*}
we find that $B'_{(\vep^2-(x_n-\theta\vep)^2)^{1/2}}(x')\subset B'_{((M\vep)^2-(x_n+\theta\vep)^2)^{1/2}}(x')$, as desired.
The last case is when $-(1+\theta)\vep < x_n \leq-\theta \vep$. Here it is clear that, since $M>1$,
\begin{align*}
\{y':(y',\psi(y'))\in B_\vep(x)\} &\subset \{y':(y',-\theta\vep)\in B_\vep(x)\} \\
&=B'_{(\vep^2-(x_n+\theta\vep)^2)^{1/2}}(x') \\
& \subset B'_{((M\vep)^2-(x_n+\theta\vep)^2)^{1/2}}(x').
\end{align*}
This concludes the proof of \eqref{eq:1}.

For \eqref{eq:2}, notice that if $x_n\geq(1-\theta)\vep$ then the inclusion follows as $\{y':(y',-\theta\vep)\in B_\vep(x)\}=\varnothing$.
We therefore assume that $-(1+\theta)\vep< x_n<(1-\theta)\vep$. If $x_n\geq-\theta\vep$ then 
\begin{align*}
\{y':(y',\psi(y'))\in B_{M\vep}(x))\} &\supset\{y':(y',-\theta\vep)\in B_{M\vep}(x)\} \\
&=B'_{((M\vep)^2-(x_n+\theta\vep)^2)^{1/2}}(x') \\
&\supset B'_{(\vep^2-(x_n+\theta\vep)^2)^{1/2}}(x') 
\end{align*}
because $M>1$. If $-(1+\theta)\vep<x_n < - \theta\vep$ then
$$\{y':(y',\psi(y'))\in B_{M\vep}(x))\}\supset\{y':(y',\theta\vep)\in B_{M\vep}(x)\}=B'_{((M\vep)^2-(x_n-\theta\vep)^2)^{1/2}}(x')$$
and
\begin{align*}
(M\vep)^2-(x_n-\theta\vep)^2& - (\vep^2-(x_n+\theta\vep)^2) \\
&= (1+2\theta)^2\vep^2 -(x_n^2 - 2\theta\vep x_n + (\theta\vep)^2)- \vep^2 + (x_n^2 +2\theta\vep x_n + (\theta\vep)^2)\\
&= 4\theta\vep^2 + 4\theta^2\vep^2 +4\theta\vep x_n \geq 0
\end{align*}
so that $B'_{((M\vep)^2-(x_n-\theta\vep)^2)^{1/2}}(x')\supset B'_{(\vep^2-(x_n+\theta\vep)^2)^{1/2}}(x')$, as desired.
Whence, \eqref{eq:2} holds.
\end{proof}

\begin{proof}[Proof of Theorem \ref{thm:stability}]
Let $M=1+2\theta>1$. By Corollary~\ref{cor:flat1} with $a=-\theta\vep$, $c_0=M^n (1+\delta)(1-\vep)^{-1}$
and $B_{3/4}$ in place of $B_1$, there is a unique classical solution $\underline{w}$ to the flat transmission problem
$$\begin{cases}
\Delta \underline{w} =M^n (1+\delta)(1-\vep)^{-1}\,dH^{n-1}\big|_{T_{-\theta\vep}}&\hbox{in}~B_{3/4}\\
\underline{w} = u&\hbox{on}~\partial B_{3/4}.
\end{cases}$$
Moreover, by subtracting the harmonic function $h$ in $B_1$ such that $h=u$
on $\partial B_1$ and applying similar arguments as in the proof of Theorem~\ref{thm:existence}, 
it can be seen that $u\in C^{0,\gamma}(\overline{B_{3/4}})$
with
$$
\|u\|_{C^{0,\gamma}(\overline{B_{3/4}})}\leq C(n,\gamma,\Gamma)(\|u\|_{L^\infty(B_1)}+\|g\|_{L^\infty(\Gamma)}) 
\leq C 
$$
where $C=C(n,\gamma)>0$ because $\Gamma$ is $\vep$-horizontal, $\|u\|_{L^\infty(B_1)}\leq 1$, and $\|g-1\|_{L^\infty(\Gamma)}\leq\delta$.
Hence, by Corollary \ref{cor:flat1} with $B_{3/4}$ in place of $B_1$, $\underline{w}\in C^\infty(B_{3/4}\setminus T_{-\theta\vep})\cap C^{0,\gamma}(\overline{B_{3/4}})$ with 
$$
\|\underline w\|_{C^{0,\gamma}(\overline{B_{3/4}})}\leq C(n,\gamma)(c_0+\|u\|_{C^{0,\gamma}(\overline{B_{3/4}})}) \leq C
$$
where $C=C(n,\gamma)>0$. 

Define the averages
$$u_\vep(x)=\frac{1}{|B_\vep|}\int_{B_\vep(x)}u(y)\,dy\qquad\hbox{for}~x\in B_{3/4-\vep}\subset B_{3/4}$$
and
$$\underline{w}_{M\vep}(x)=\frac{1}{|B_{M\vep}|}\int_{B_{M\vep}(x)}\underline{w}(y)\,dy
\qquad\hbox{for}~x\in B_{3/4-M\vep}\subset B_{3/4}.$$
By Proposition~\ref{lem:averages}$(iii)$, $\Delta u_\vep(x)=g_\vep(x)$ for every $x\in B_{3/4-\vep}$, and
$$
\Delta\underline{w}_{M\vep}(x) = \frac{1}{|B_{M\vep}|}\int_{B_{M\vep}(x)\cap T_{-\theta\vep}}M^n (1+\delta)(1-\vep)^{-1}\,dH^{n-1}
\qquad\hbox{for}~x\in B_{3/4-M\vep}.$$
In addition, notice that
$$\supp(\Delta u_\vep)\subset\{x\in B_{3/4-\vep}:\dist(x,\Gamma)<\vep\}$$
and
$$\supp(\Delta \underline{w}_{M\vep})\subset\{x\in B_{3/4-M\vep}:|x_n|<M\vep\}.$$
Since $\Gamma$ is $\theta\vep$-flat in $B_1$ and $M=1+2\theta$ it follows that
$$\supp(\Delta u_\vep)\subset\supp(\Delta \underline{w}_{M\vep}).$$
Let us first show that
$$\Delta \underline{w}_{M\vep}\geq\Delta u_\vep\qquad\hbox{in}~B_{3/4-M\vep}.$$
If $x\notin\supp(g_\vep)$ there is nothing to prove because $\Delta \underline{w}_{M\vep}\geq0$ in $B_{3/4-M\vep}$.
Let us then take $x\in B_{3/4-M\vep}$ such that $\dist(x,\Gamma)<\vep$.
Using that $0<g\leq 1+\delta$, $\Gamma$ is $\vep$-horizontal and \eqref{eq:1} in Lemma~\ref{lem:medidas}, we get
\begin{align*}
\Delta \underline{w}_{M\vep}(x) &= \frac{1}{M^n |B_{\vep}|}\int_{B_{M\vep}(x)\cap T_{-\theta\vep}} M^n (1+\delta)(1-\vep)^{-1} \, dH^{n-1}\\
&\geq \frac{1}{ |B_{\vep}|}\int_{\{y':(y',-\theta\vep)\in B_{M\vep}(x)\}}g(y',\psi(y'))\sqrt{1+|\nabla'\psi(y')|^2}\,dy' \\
&\geq \frac{1}{|B_{\vep}|}\int_{\{y':(y',\psi(y'))\in B_\vep(x)\}}g(y',\psi(y'))\sqrt{1+|\nabla'\psi(y')|^2}\,dy' \\
&=\frac{1}{|B_\vep|}\int_{B_\vep(x)\cap\Gamma}g\,dH^{n-1}=\Delta u_\vep(x).
\end{align*}
We also have
$$\underline{w}_{M\vep}\leq u_\vep+C\vep^\gamma\qquad\hbox{on}~\partial B_{3/4-M\vep}$$
for some $C=C(n,\gamma)>0$.
Indeed, fix any $x\in\partial B_{3/4-M\vep}$, and let $z\in\partial B_{3/4}$ be such that $\dist(x,\partial B_{3/4}) = |x-z|= M\vep$.
By using that $\underline{w},u\in C^{0,\gamma}(\overline{B_{3/4}})$ and $\underline{w}=u$ on $\partial B_{3/4}$,
\begin{equation}\label{eq:promedios}
\begin{aligned}
\underline{w}_{M\vep}(x)-u_\vep(x) 
&=(\underline{w}_{M\vep}(x)-\underline{w}(x))+(\underline{w}(x)-\underline{w}(z))+(u(z)-u(x))+(u(x)-u_\vep(x)) \\
&\leq\frac{1}{|B_{M\vep}|}\int_{B_{M\vep}(x)}|\underline{w}(y)-\underline{w}(x)|\,dy
+ ([\underline{w}]_{C^{0,\gamma}(\overline{B_{3/4}})} +[u]_{C^{0,\gamma}(\overline{B_{3/4}})} )|x-z|^\gamma \\
&\quad+\frac{1}{|B_{\vep}|}\int_{B_{\vep}(x)}|u(y)-u(x)|\,dy \leq C\vep^\gamma
\end{aligned}
\end{equation}
where $C=C(n,\gamma)>0$.
Hence, by the maximum principle, $\underline{w}_{M\vep}-u_{\vep}\leq C\vep^\gamma$ in  $B_{3/4-M\vep}$.
Consequently, by arguing similarly as in \eqref{eq:promedios}, it follows that, for some $C=C(n,\gamma)>0$,
\begin{equation}\label{eq:A}
\underline{w}-u\leq C\vep^\gamma\qquad\hbox{in}~B_{3/4-M\vep}.
\end{equation}

Secondly, consider the classical solution $\bar{w}$ to the flat transmission problem
$$\begin{cases}
\Delta \bar{w} = M^{-n} (1-\delta) \,dH^{n-1}\big|_{T_{-\theta\vep}}&\hbox{in}~B_{3/4}\\
\bar{w} = u&\hbox{on}~\partial B_{3/4}
\end{cases}$$
and the corresponding averages $\bar{w}_\vep$ and $u_{M\vep}$ of $\bar{w}$ and $u$, respectively.
Since $g\geq1-\delta$, by \eqref{eq:2} in Lemma \ref{lem:medidas} we find that
\begin{align*}
\Delta \bar{w}_\vep (x) &= \frac{1}{|B_\vep|} \int_{B_\vep(x)\cap T_{-\theta\vep}} M^{-n} (1-\delta) \, dH^{n-1}\\
 &\leq \frac{1}{|B_{M\vep}|} \int_{\{y':(y',-\theta\vep)\in B_\vep(x)\}} g(y',\psi(y'))\sqrt{1+|\nabla'\psi(y')|^2}\,dy' \\
 &\leq \frac{1}{|B_{M\vep}|} \int_{\{y' : (y',\psi(y'))\in B_{M\vep}(x)\}} g(y',\psi(y'))\sqrt{1+|\nabla'\psi(y')|^2}\,dy'\\
 &=  \frac{1}{|B_{M\vep}|} \int_{B_{M\vep}(x)\cap \Gamma} g\, dH^{n-1}=\Delta u_{M\vep}(x).
\end{align*}
By using parallel arguments to those in \eqref{eq:promedios} we also get that, for some $C=C(n,\gamma)>0$,
\begin{equation}\label{eq:B}
u - \bar{w} \leq C\vep^\gamma\qquad\hbox{in}~B_{3/4-M\vep}.
\end{equation}

Define
${w = \frac{\underline{w}+\bar{w}}{2}.}$
By \eqref{eq:A} and \eqref{eq:B},
$$u-w \leq \bar{w} +  C\vep^\gamma - \frac{\underline{w}+\bar{w}}{2} = \frac{\bar{w}-\underline{w}}{2} +  C\vep^\gamma$$
and
$$u-w  \leq \underline{w} -  C\vep^\gamma -  \frac{\underline{w}+\bar{w}}{2} =  \frac{\underline{w}-\bar{w}}{2} - C\vep^\gamma.$$
Hence,
$$\|u-w\|_{L^\infty(B_{1/2})} \leq \frac{1}{2}\|\bar{w}-\underline{w} \|_{L^\infty(B_{1/2})}+ C\vep^\gamma$$
where $C=C(n,\gamma)>0$. Since
$$\begin{cases}
\Delta(\bar{w}-\underline{w})=[M^{-n} (1-\delta)-M^n (1+\delta)(1-\vep)^{-1}]\,dH^{n-1}\big|_{T_{-\theta\vep}}&\hbox{in}~B_{3/4}\\
\bar{w}-\underline{w}=0&\hbox{on}~\partial B_{3/4}
\end{cases}$$
by Theorem \ref{thm:flat}, 
$$
\|\bar{w}-\underline{w}\|_{L^\infty(B_{1/2})} \leq C |M^n (1+\delta)(1-\vep)^{-1} - M^{-n} (1-\delta)|\leq C(\theta+\delta+\vep)$$
for some $C=C(n)>0$. Therefore, 
\begin{equation} \label{eq:est1}
\|u-w\|_{L^\infty(B_{1/2})}\leq C(\theta+\delta+\vep^\gamma)
\end{equation}
for some $C=C(n,\gamma)>0$. Also, $\Delta w = (1+\eta)\, dH^{n-1}\big|_{T_{-\theta\vep}}$ where
$$1+\eta = \frac{M^n (1+\delta)(1-\vep)^{-1} + M^{-n} (1-\delta)}{2}.$$
Observe that, since $0<\theta,\vep<1/2$, $0<\delta<1$, it follows that
\begin{equation}\label{eq:eta}
\begin{aligned}
|\eta| &= \frac{|M^{2n}(1+\delta) + (1-\delta)(1-\vep) -2(1-\vep)M^n|}{2(1-\vep)M^n} \\
&\leq C\big(|(1+2\theta)^{2n} +1 - 2(1+2\theta)^n| + \delta + \vep\big)\leq C(\theta+\delta+\vep)
\end{aligned}
\end{equation}
where $C=C(n)>0$.

Let $v\in C^\infty(B_{3/4} \setminus T_{-\theta\vep})\cap C^{0,\gamma}(\overline{B_{3/4}})$ be the solution to
$$\begin{cases}
{\Delta v= dH^{n-1}\big|_{T_{-\theta\vep}}}&\hbox{in}~B_{3/4}\\
v=u&\hbox{on}~\partial B_{3/4}
\end{cases}$$
see Corollary~\ref{cor:flat1}. Then $v-w$ solves
$$\begin{cases}
\Delta(v-w)=\eta\,dH^{n-1}\big|_{T_{-\theta\vep}}&\hbox{in}~B_{3/4}\\
v-w=0&\hbox{on}~\partial B_{3/4}.
\end{cases}$$
Therefore, by \eqref{eq:eta},
\begin{equation} \label{eq:est2}
\|v-w\|_{L^\infty(B_{3/4})}\leq C|\eta|\leq C(\theta+\delta+\vep)
\end{equation}
where $C=C(n)>0$. From \eqref{eq:est1} and \eqref{eq:est2} the estimate on the statement is proved.
\end{proof}

\begin{rem}[Divergence form equations]
Recall that our proof of Theorem \ref{thm:stability} is self-contained,
based on the mean value property for harmonic functions
and the maximum principle. In view of recently developed mean value formulas
for solutions to divergence form elliptic equations
by Blank--Hao \cite{Blank-Hao}, the natural question of extending our geometric techniques
to transmission problems for divergence form elliptic equations arise. For this case, our maximum
principle techniques must be replaced by energy methods. More importantly,
not much is known about the geometry of the mean value sets from \cite{Blank-Hao},
so it is not clear at all how to mimic geometric arguments such as those in Lemma \ref{lem:medidas}.
\end{rem}

\begin{rem}[Nondivergence form equations]
The second natural question would be to extend our methods to transmission problems with
nondivergence form elliptic equations, where the maximum principle is a more adequate tool.
In this situation, not only there are no useful mean value formulas available,
but also the notion of distributional solution
we consider in this paper does not apply anymore. The first step would be to prove existence, uniqueness, and
some initial regularity of viscosity solutions. This is an open problem.
\end{rem}

\section{Proof of Theorem \ref{thm:boundary}} \label{Section:boundary}

Throughout this section, $\Gamma$ is an interface in $B_1$ given by the graph of a function $x_n=\psi(x'):T\to\R$. Thus, we can write
$B_1=\Omega_1\cup\Gamma\cup\Omega_2$, where $\Omega_1=\{x=(x',x_n)\in B_1:x_n>\psi(x')\}$.
We also assume that $0\in\Gamma$.

\subsection{Preliminary lemmas}

\begin{lem}\label{lem:basic}
Let $\Gamma=\{(y',\psi(y')):y'\in B_1'\}$, where $\psi$ is a Lipschitz function.
Given $0<\alpha,\gamma<1$, there exist constants $C_0>0$, $0<\lambda<1/2$, $0<\theta, \delta,\vep<\lambda$
depending only on $n$, $\alpha$ and $\gamma$, such that for any $u\in C(\overline{B_1})$ satisfying 
$$\begin{cases}
\Delta u = g \, dH^{n-1} \big |_{\Gamma} &\hbox{in}~B_1\\
|u|\leq 1&\hbox{in}~B_1\\
|g-1|\leq \delta & \hbox{on}~\Gamma
\end{cases}$$
if $\Gamma$ is $\theta\vep$-flat and $\vep$-horizontal in $B_1$, then there are linear polynomials $P_1(x) = A\cdot x + B$ and $Q_1(x)=C\cdot x + B$,  with $A,C\in \R^n$, $B\in \R$, and $|A|+|B|+|C|\leq C_0$, such that
\begin{align*}
|u_1(x)-P_1(x)| &\leq \lambda^{1+\alpha}\qquad\hbox{for all}~x\in \Omega_1\cap B_\lambda\\
|u_2(x)-Q_1(x)|&\leq \lambda^{1+\alpha}\qquad\hbox{for all}~x\in \Omega_2\cap B_\lambda.
\end{align*}
Moreover, $\nabla' P_1=\nabla'Q_1$ and $(P_1)_{x_n} - (Q_1)_{x_n} = 1$.
\end{lem}

\begin{proof}
Fix $0<\theta, \delta,\vep<\lambda<1/2$ to be chosen later. Consider the solutions 
\begin{align*}
\underline{v} &= \underline{v}^+ \chi_{\overline{B_{3/4}}\cap T_{-\theta\vep}^+}+\underline{v}^- \chi_{\overline{B_{3/4}}\cap T_{-\theta\vep}^-} \\
\bar{v} &= \bar{v}^+ \chi_{\overline{B_{3/4}}\cap T_{\theta\vep}^+} + \bar{v}^- \chi_{\overline{B_{3/4}}\cap T_{\theta\vep}^-} 
\end{align*}
to the flat transmission problems given in Theorem~\ref{thm:stability}, and Remark~\ref{rem:above}, respectively. By Corollary~\ref{cor:flat1} with $k=2$,
$$\| \underline v^+ \|_{C^{2,\alpha}(\overline{B_{1/2}}\cap T_{-\theta\vep}^+)} 
+\| \bar v^- \|_{C^{2,\alpha}(\overline{B_{1/2}}\cap T_{\theta\vep}^-)} 
\leq C\big( 1 +  \| u\|_{L^\infty(B_1)} \big)\leq C_0$$
for some $C_0=C_0(n,\alpha)>0$. In particular,
$$
| \underline v(0)| + |\nabla \underline v(0)|  +
 | \bar v(0)|+ |\nabla \bar v(0)| \leq C_0.
$$
Let $h$ be the harmonic function in $B_{3/4}$ such that $h=u$ on $\partial B_{3/4}$.
Define 
\begin{align*}
P_1(x) &= \underline v(0) + \nabla \underline v(0) \cdot x +\big[\tfrac{1}{2}-\underline{v}_{x_n}(0)+h_{x_n}(0)\big] x_n \\ 
Q_1(x) &= \bar v(0) + \nabla \bar v(0) \cdot x + \big[-\tfrac{1}{2}-\bar{v}_{x_n}(0)+h_{x_n}(0)\big] x_n.
\end{align*}
Then $P_1$ and $Q_1$ are small perturbations of the linear parts of $\underline v$ and $\bar v$ at the origin, respectively.
To see this, first note that the functions $\underline{v}(x',x_n) - h(x',x_n)$
and $\bar{v}(x', -x_n) - h(x',- x_n)$ satisfy the same transmission problem on $T_{-\theta\vep}$ with zero data on $\partial B_{3/4}$. By uniqueness, 
$$
\underline{v}(x',x_n) - h(x',x_n) = \bar{v}(x', -x_n) - h(x',- x_n) \qquad\hbox{for all}~x\in \overline{B_{3/4}}.
$$
In particular, $\underline{v}(x',0) = \bar{v}(x', 0)$, $\nabla'\underline{v}(x',0) =  \nabla' \bar {v}(x',0)$, and thus, $P_1(0)=Q_1(0)$, and $\nabla'P_1=\nabla'\underline{v}(0)=\nabla' \bar {v}(0)=\nabla'Q_1$. Clearly, $(P_1)_{x_n} - (Q_1)_{x_n} =1$. Moreover,
$$
\underline{v}_{x_n}(x',0) -h_{x_n}(x',0) = - \bar{v}_{x_n}(x', 0)+h_{x_n}(x',0)
$$ 
and thus, $\big|\tfrac{1}{2}-\underline{v}_{x_n}(0)+h_{x_n}(0)\big| = \big|-\tfrac{1}{2}-\bar{v}_{x_n}(0)+h_{x_n}(0)\big|$.
Let us show that
\begin{equation} \label{eq:smallerror}
\big|\tfrac{1}{2}-\underline{v}_{x_n}(0)+h_{x_n}(0)\big| \leq D (\theta\vep)^\gamma
\end{equation}
 for some $D=D(n)>0$. Recall that by the construction of $\underline{v}$ in Corollary~\ref{cor:flat1}, we can write $\underline{v}=w-H$, where $w\in C^\infty(B_{4,-\theta\vep}\setminus T_{-\theta\vep})\cap C^{0,\gamma}(\overline{B_{4,-\theta\vep}})$ is the harmonic function in $B_{4,-\theta\vep}$ such that $w=0$ on $\partial B_{4,-\theta\vep}$, and $H$ is the harmonic function in $B_1$, with $H=w-u$ on $\partial B_{3/4}$.
 Then
 $$
\big|\tfrac{1}{2}-\underline{v}_{x_n}(0)+h_{x_n}(0)\big|  \leq \big|w_{x_n}(0)- \tfrac{1}{2}\big| + |(H+h)_{x_n}(0)|.
 $$
 In particular, $w_{x_n}(0)=w^+_{x_n}(0)$, where $w^+$ is the harmonic function in $B_{4,-\theta\vep}^+$ such that $w=0$ on $\partial B_{4,-\theta\vep}^+\setminus T_{-\theta\vep}$, and $w^+_{x_n}=\tfrac{1}{2}$ on $T_{-\theta\vep}$. By the mean value theorem,
 $$
 w_{x_n}(0)- \tfrac{1}{2} = w_{x_n}^+(0',0) - w_{x_n}^+(0',-\theta\vep)=  w_{x_nx_n}^+(0', \xi) \theta\vep
 $$
 for some $-\theta \vep\leq\xi\leq0$. Moreover, by Theorem~\ref{thm:flat}, $\| w^+ \|_{C^{2,\alpha}(\overline{B_{2,-\theta\vep}^+})} \leq D_0$, for some constant $D_0=D_0(n)>0$. Hence, 
 $$
 | w_{x_n}(0) - \tfrac{1}{2}|\leq D_0\theta\vep.
 $$ 
Next, note that $H+h$ is harmonic in $B_1$, and $H+h=w$ on $\partial B_{3/4}$. Consider the harmonic function $\phi$ in $B_{3/4-\theta\vep,-\theta\vep}$ such that $\phi = w$ on $B_{3/4-\theta\vep,-\theta\vep}$. Observe that $B_{3/4-\theta\vep,-\theta\vep}\subset B_{3/4}$. Since $w$ is symmetric with respect to the plane $T_{-\theta\vep}$, it follows that $\phi_{x_n}(x',-\theta\vep)=0$ for any $(x',-\theta\vep)\in B_{3/4-\theta\vep,-\theta\vep}$. Therefore, 
 $| \phi_{x_n}(0)| \leq D_0 \theta\vep$. By interior estimates, the maximum principle, and the facts that $w\in C^{0,\gamma}(\overline{B_{3/4}})$ and $\dist(\partial B_{3/4}, \partial B_{3/4-\theta\vep,-\theta\vep})\leq 2\theta \vep$,
 $$
|(H+h)_{x_n}(0)- \phi_{x_n}(0)| \leq D_1 \|(H+h) - w\|_{L^\infty(\partial B_{3/4-\theta\vep,-\theta\vep})} \leq D_1 (\theta \vep)^\gamma
 $$
 for some $D_1=D_1(n)>0$, and thus,
 $$
 |(H+h)_{x_n}(0)| \leq D_1 (\theta \vep)^\gamma+ | \phi_{x_n}(0)| \leq D_1(\theta \vep)^\gamma +D_0 \theta\vep \leq D(\theta \vep)^\gamma
 $$
 for some $D=D(n)>0$. Therefore, \eqref{eq:smallerror} holds. 

If $x\in \Omega_1\cap B_{1/2}$, by Theorem~\ref{thm:stability} and \eqref{eq:smallerror}, there are constants $C,D>0$, depending only on $n$, such that
\begin{align*}
|u_1(x) - P_1(x)| & \leq |u(x) - \underline v(x)| + |\underline v(x) - P_1(x)| \\
&\leq  |u(x) -  \underline v(x)| + |\underline v(x)-\underline v(0) - \nabla \underline v(0) | + \big|\tfrac{1}{2}-\underline{v}_{x_n}(0)+h_{x_n}(0)\big| |x_n|\\
&\leq C(\theta+\delta+\vep^\gamma) + \Vert D^2 \underline v\Vert_{L^\infty(\Omega_1\cap B_{1/2})}|x|^2+D(\theta\vep)^\gamma|x_n|\\
&\leq C(\theta+\delta+\vep^\gamma) + C_0|x|^2+ D(\theta\vep)^\gamma|x_n|
\end{align*}
Similarly, if $x\in \Omega_2 \cap B_{1/2}$,
\begin{align*}
|u_2(x) - Q_1(x)| & \leq C(\theta+\delta+\vep^\gamma) + C_0|x|^2+ D(\theta\vep)^\gamma|x_n|.
\end{align*}
First, choose $0<\lambda<1/2$ such that
$$C_0|x|^2\leq \frac{\lambda^{1+\alpha}}{2}\qquad\hbox{for all}~x\in B_\lambda.$$
Then, choose $0<\theta,\delta,\vep<\lambda$ such that
$$ C(\theta+\delta+\vep^\gamma)+ D(\theta\vep)^\gamma \lambda \leq  \frac{\lambda^{1+\alpha}}{2}.$$
\end{proof}

\begin{lem} \label{lem:harmonic}
Let $\Gamma=\{(y',\psi(y')):y'\in B_1'\}$, where $\psi$ is a Lipschitz function.
Given $0<\alpha<1$, there exist ${C}_0>0$, $0< \lambda<1/2$, and $0<\delta<1$, depending only on $n$ and $\alpha$, such that for a distributional solution ${u}\in C(\overline{B_1})$ to 
$$\begin{cases}
\Delta {u} = {g} \, dH^{n-1} \big |_{ \Gamma} &\hbox{in}~B_1\\
|u|\leq 1&\hbox{in}~B_1\\
|{g} | \leq \delta & \hbox{on}~\Gamma\cap B_{3/4}
\end{cases}$$
there is a linear polynomial ${P}(x) = {A}\cdot x + {B}$, with ${A}\in \R^n$, ${B}\in \R$ and $|{A}|+|{B}|\leq {C}_0$, such that
$$
|{u}(x)-{P}(x)|\leq \lambda^{1+\alpha} \qquad\hbox{for all}~x\in B_{\lambda}.
$$
\end{lem}

\begin{proof}
Fix $\lambda,  \delta>0$ to be determined. Let ${v}$ be the harmonic function in $B_{3/4}$ such that
$ v = u$ on $\partial B_{3/4}$. Then, the difference $w= u-  v$ is the distributional solution to 
$$\begin{cases}
\Delta w = {g} \, dH^{n-1} \big |_{ \Gamma}&\hbox{in}~B_{3/4}\\
w=0&\hbox{on}~\partial B_{3/4}.
\end{cases}$$
Moreover,
$\| w\|_{L^\infty(B_{3/4})} \leq C \|  g\|_{L^\infty (\Gamma\cap B_{3/4})} \leq C  \delta$,
where $C=C(n,\Gamma)>0$. Define $ P(x) = {v}(0) + \nabla {v}(0)\cdot x$. By interior estimates and the maximum principle, we have
\begin{align*}
\| D^j  v\|_{L^\infty(B_{1/2})}& \leq  C_0 \| v\|_{L^\infty(B_{3/4})} \leq {C}_0 \qquad\hbox{for all}~j \geq 0 
\end{align*}
where $ C_0= C_0(n,j)>0$. Hence, for $x\in B_{\lambda}$, with $0<\lambda<1/2$, we get
\begin{align*}
| u(x) -  P(x) | & \leq | u(x) -  v(x)| + | v(x) -  P(x)|\\
&\leq C  \delta + \| D^2  v\|_{L^\infty(B_{1/2})} |x|^2\\
&\leq C \delta +  C_0 \lambda^2.
\end{align*}
First, choose $0<\lambda <1/2$, such that $ C_0 \lambda^2 \leq \lambda^{1+\alpha}/2$. Then choose $0<\delta <1$ such that $C\delta \leq \lambda^{1+\alpha}/2$.
\end{proof}

\subsection{Proof of Theorem~\ref{thm:boundary}} 

Fix $0<\alpha, \gamma<1$. Let $C_0, \lambda, \theta, \vep, \delta>0$ be the minimum of the constants given in Lemma~\ref{lem:basic} and Lemma~\ref{lem:harmonic}. Let $0<\delta_0<\min \big \{\delta, \theta\vep,\tfrac{\lambda^{1+\alpha}}{2}\big\}$.
First, we normalize the problem. Recall that we are assuming that $0\in\Gamma$, that is, $\psi(0')=0$.

\begin{enumerate}[$(i)$]
\item By rotation, we can assume that $\nu(0)=e_n$. In particular, $\nabla'\psi(0')=0'$.  
\item If $g(0)\neq 0$, we can suppose that $g(0)=1$. Indeed, we consider $v={u}/{g(0)}$. The case $g(0)=0$ will be addressed at the end. 
\item Assume that $\| u \|_{L^\infty(B_1)}\leq 1$, and that
$$
[g]_{C^{0,\alpha}(0)} = \sup_{x\in \Gamma\cap B_1,\,  x\neq 0} \frac{|g(x)-g(0)|}{|x|^\alpha} \leq {\delta_0}.
$$
Indeed, one can consider 
$$
v = {\delta_0}\, \frac{u}{\| u \|_{L^\infty(B_1)}+[g]_{C^{0,\alpha}(0)}}.
$$
\item Also, we let $[\psi]_{C^{1,\alpha}(0)} \leq [\psi]_{C^{1,\alpha}(B_1')}\leq \delta_0$. Recall that
$$
[\psi]_{C^{1,\alpha}(0)}  = \sup_{x'\in B_1', \, x'\neq 0'}  \frac{|\nabla' \psi(x')-\nabla'\psi(0')|}{|x'|^{\alpha}} =  \sup_{x'\in B_1', \, x'\neq 0'} \frac{|\nabla' \psi(x')|}{|x'|^{\alpha}}.
$$
Then, for this normalization one can take
$$
\phi = {\delta_0} \, \frac{\psi}{[\psi]_{C^{1,\alpha}(B_1')} }.
$$
\end{enumerate}
 
We make an abuse of notation and call the solution, the interface, the parametrization
and the right hand side as in the statement, namely, $u$, $\Gamma$, $\psi$, and $g$, respectively.

It is enough to prove the following. \medskip

\noindent{\bf Claim.} {\it For all $k \geq 1$, there exist linear polynomials $P_k=A_k \cdot x + B_k$ and $Q_k=C_k\cdot x + B_k$ such that
\begin{align*}
\lambda^{k} |A_{k+1}-A_{k}| + \lambda^{k} |C_{k+1}-C_{k}| + |B_{k+1}-B_{k}|  \leq C_0 \lambda^{k(1+\alpha)}
\end{align*}
where $C_0=C_0(n,\alpha)>0$, and such that
\begin{align*}
|u_1(x)-P_k(x)| &\leq \lambda^{k(1+\alpha)}\qquad \qquad~\hbox{for all}~x\in {\Omega_1}\cap B_{\lambda^k}\\
|u_2(x)-Q_k(x)| &\leq \lambda^{k(1+\alpha)}\qquad \qquad~\hbox{for all}~x\in {\Omega_2}\cap B_{\lambda^k}.
\end{align*}
Moreover, $\nabla'P_k=\nabla'Q_k$ and  $(P_k)_{x_n}-(Q_k)_{x_n}=1$.\medskip}

We prove the claim by induction. Let us start with the case $k=1$. By the normalization, $u$, $\Gamma$ and $g$ satisfy the assumptions on Lemma~\ref{lem:basic}. Indeed, by $(i)$ and $(iv)$, for any $(x',x_n)\in \Gamma$,
\begin{align*}
|x_n| & = |\psi(x')| = |\psi(x')-\psi(0') - \nabla'\psi(0')\cdot x'|\leq [\psi]_{C^{1,\alpha}(0)}\leq {\delta_0} \leq \theta\vep.
\end{align*}
Also, $1\leq (1+|\nabla'\psi(x')|^2)^{1/2} \leq (1+\delta_0^2)^{1/2}\leq (1-\vep)^{-1}$.
Moreover, by $(iii)$, 
$$
|g(x)-1| = |g(x) - g(0)| \leq [g]_{C^{0,\alpha}(0)}|x|^\alpha \leq \delta_0 \leq \delta \qquad\hbox{for any}~x \in \Gamma.
$$
Hence, by Lemma~\ref{lem:basic}, there are linear polynomials $P_1(x) = A_1\cdot x + B_1$, and $Q_1(x)=C_1\cdot x + B_1$,  with $A_1,C_1\in \R^n$, $B_1\in \R$, and $|A_1|+|B_1|+|C_1|\leq C_0$, such that
\begin{align*}
|u_1(x)-P_1(x)| &\leq \lambda^{1+\alpha}\qquad\hbox{for all}~x\in \Omega_1\cap B_\lambda\\
|u_2(x)-Q_1(x)|&\leq \lambda^{1+\alpha}\qquad\hbox{for all}~x\in \Omega_2\cap B_\lambda.
\end{align*}
Moreover, $\nabla' P_1=\nabla'Q_1$, and $(P_1)_{x_n} - (Q_1)_{x_n} = 1$.

For the induction step, assume that the claim holds for some $k\geq 1$, and let $P_k$ and $Q_k$ be such polynomials. Denote by 
\begin{align*}
{\Omega}_{i,\lambda^k} &=\{ x\in B_1 : \lambda^k x\in \Omega_i\} \quad \text{for }\ i=1,2\\
\Gamma_{\lambda^k} & =\{ x \in B_1 : \lambda^k  x \in \Gamma \}.
\end{align*}
Note that if  $\psi_{\lambda^k}$ is a parametrization of $\Gamma_{\lambda^k} $ in $B_1'$, then
$
\psi_{\lambda^k}(x') = \lambda^{-k}\psi(\lambda^k x').
$
In particular, $\nabla' \psi_{\lambda^k} (x') =\nabla' \psi(\lambda^k x)$, and thus, for $x\in \Gamma_{\lambda^k}$, if $\nu_{\lambda^k}(x)$ is the normal vector on $x$ pointing at $\Omega_{\lambda^k,1}$, then  $\nu_{\lambda^k}(x) = \nu(\lambda^k x)$.
Define $\mathcal{P}_k= P_k \chi_{\Omega_1} + Q_k \chi_{\Omega_2}$.
Consider the rescaled function
\begin{align}\label{eq:rescale}
w(x) = \frac{u(\lambda^k x) - \mathcal{P}_k(\lambda^k x)}{\lambda^{k(1+\alpha)}}\qquad\hbox{for}~x\in \overline{B_1}.
\end{align}
By the induction hypothesis, 
$\| w \|_{L^\infty(B_1)}  \leq 1.$
Notice that $w$ is a piecewise continuous function with a jump discontinuity on $\Gamma_{\lambda^k}$. In fact, if
$$ 
w_1 = w \big|_{\overline\Omega_{1,\lambda^k}}, \quad w_2 = w \big|_{\overline\Omega_{2,\lambda^k}}
$$
then for $x\in \Gamma_{\lambda^k}$, by the normalization $(iv)$, and the induction hypothesis, we have
\begin{align} \label{eq:jump}
|(w_1 - w_2)(x)| &= \frac{|Q_k(\lambda^k x)-P_k(\lambda^k x)|}{\lambda^{k(1+\alpha)}} = \lambda^{-k \alpha} |x_n|\\\nonumber
&\leq   \lambda^{-k \alpha} \sup_{x \in \Gamma_{\lambda^k}} |x_n|\\\nonumber
 &\leq   \sup_{x'\in B_1'} \frac{|\psi_{\lambda^k}(x')|}{ \lambda^{k \alpha} } \leq [\psi]_{C^{1,\alpha}(0)} \leq  \delta_0. 
\end{align}

Let $v=v_1\chi_{\overline{\Omega}_{1,\lambda^k}}+v_2\chi_{\overline{\Omega}_{2,\lambda^k}}$, where $v_1$ and $v_2$ are the solutions to
$$\begin{cases}
\Delta v_i=0&\hbox{in}~\Omega_{i,\lambda^k}\\
v_i=w_i&\hbox{on}~\partial\Omega_{i,\lambda^k}\setminus \Gamma_{\lambda^k}\\
v_i=\frac{w_1+w_2}{2}&\hbox{on}~\Gamma_{\lambda^k}
\end{cases}$$
for $i=1,2$. Then $v\in C(\overline{B_1})$ and, by the maximum principle, $\| v\|_{L^\infty(B_1)}\leq \| w\|_{L^\infty(B_1)}\leq 1$. Moreover,
\begin{equation}\label{eq:correction}
\begin{cases} 
\Delta (v_i-w_i)=0&\hbox{in}~\Omega_{i,\lambda^k}\\
v_i-w_i=0&\hbox{on}~\partial\Omega_{i,\lambda^k}\setminus\Gamma_{\lambda^k}\\
v_i-w_i=(-1)^{i} \frac{w_1-w_2}{2}&\hbox{on}~\Gamma_{\lambda^k}.
\end{cases}
\end{equation}
By the maximum principle and \eqref{eq:jump} it follows that
\begin{equation}\label{eq:vclosetow}
\begin{aligned} 
\|v-w\|_{L^\infty(B_1)} & \leq \|v_1-w_1\|_{L^\infty(\Omega_{1,\lambda^k})}+ 
\|v_2-w_2\|_{L^\infty(\Omega_{2,\lambda^k})} \\
&=\|w_1-w_2\|_{L^\infty(\Gamma_{\lambda^k})}
\leq  \delta_0. 
\end{aligned} 
\end{equation}

We compute the distributional Laplacian of $v$ and estimate its size. For any $\varphi\in C^\infty_c(B_1)$,
\begin{align*}
\Delta v (\varphi) & = \int_{B_1} v(x) \Delta \varphi(x) \, dx \\
&= \int_{\Omega_{1,\lambda^k}} v_1(x) \Delta \varphi(x) \, dx +  \int_{\Omega_{2,\lambda^k}} v_2(x) \Delta \varphi(x) \, dx \\
&=  \int_{\Omega_{1,\lambda^k}} (v_1-w_1)(x) \Delta \varphi(x) \, dx 
+ \int_{\Omega_{2,\lambda^k}} (v_2-w_2)(x) \Delta \varphi(x) \, dx 
+ \int_{B_1} w(x) \Delta \varphi(x) \, dx \\
&\equiv I_1 + I_2 + I_3.
\end{align*}
For $i=1,2$, by Green's formula, 
\begin{align*}
I_i &= \frac{1}{2} \int_{\Gamma_{\lambda^k}}  (w_1-w_2)(x)  \varphi_{\nu_{\lambda^k}}(x)\, dH^{n-1}+(-1)^{i+1} \int_{\Gamma_{\lambda^k}} (v_i-w_i)_{\nu_{\lambda^k}}(x) \varphi(x)\, dH^{n-1}
\end{align*}
where we recall that $\nu_{\lambda^k}$ is the unit normal vector on $\Gamma_{\lambda^k}$ pointing at $\Omega_{1,\lambda^k}$.
Note that
$$
I_3=\Delta w(\varphi) = \Delta \left(\frac{u(\lambda^k x)}{\lambda^{k(1+\alpha)}}\right)(\varphi) 
-\Delta \left(\frac{\mathcal{P}_k(\lambda^k x)}{\lambda^{k(1+\alpha)}}\right) (\varphi).
$$
Since $u$ is a distributional solution, by doing a change of variables, we get
\begin{align*}
\Delta (u(\lambda^k x))(\varphi) &= \int_{B_1} u(\lambda^k x)\Delta \varphi(x) \, dx\\
&= \lambda^{k(2-n)} \int_{B_{\lambda^{k}}} u(y) \Delta_y \varphi ( \lambda^{-k} y)\, dy\\
&= \lambda^{k(2-n)} \int_{\Gamma \cap B_{\lambda^{k}}} g(y) \varphi(\lambda^{-k} y) \, dH^{n-1}_y
= \lambda^k \int_{\Gamma_{\lambda^k}}g(\lambda^k x) \varphi(x)\, dH^{n-1}.
\end{align*}
Also, by Green's formula, the induction hypothesis and \eqref{eq:jump},
\begin{align*}
\Delta (\mathcal{P}_k(\lambda^k x))(\varphi) &= \lambda^k  \int_{\Gamma_{\lambda^k}} \big[\nabla P_k (\lambda^k x)-\nabla Q_k(\lambda^k x)\big]\cdot \nu_{\lambda^k}(x)\varphi(x) \, dH^{n-1} \\
&\qquad + \int_{\Gamma_{\lambda^k}} \big[Q_k(\lambda^k x)-P_k(\lambda^k x)\big] \varphi_{\nu_{\lambda^k}}(x) \, d H^{n-1}\\
&= \lambda^k  \int_{\Gamma_{\lambda^k}} \nu_n(\lambda^k x) \varphi(x) \, dH^{n-1} + \lambda^{k(1+\alpha)} \int_{\Gamma_{\lambda^k}} (w_1 - w_2)(x)   \varphi_{\nu_{\lambda^k}}(x) \, d H^{n-1}.
\end{align*}
Then
\begin{equation*}
I_3= \int_{\Gamma_{\lambda^k}} \tilde g(x) \varphi(x)  \, dH^{n-1} - \int_{\Gamma_{\lambda^k}} (w_1 - w_2)(x)   \varphi_{\nu_{\lambda^k}}(x) \, d H^{n-1}
\end{equation*}
where
$$
\tilde g(x) = \frac{g(\lambda^k x) - \nu_n(\lambda^k x)}{\lambda^{k\alpha}}.
$$
Therefore, for any $\varphi\in C^\infty_c(B_1)$,
\begin{align*}
\Delta v(\varphi) &= \int_{\Gamma_{\lambda^k}} \Big[  (v_1-w_1)_{\nu_{\lambda^k}}(x) 
- (v_2-w_2)_{\nu_{\lambda^k}}(x) +  \tilde{g}(x)\Big] \varphi(x)\, dH^{n-1}
\equiv\int_{ \Gamma_{\lambda^k}} \hat{g} \varphi \, dH^{n-1}.
\end{align*} 
By $C^{1,\alpha}$ boundary estimates for harmonic functions applied to \eqref{eq:correction}
and, by taking into account \eqref{eq:vclosetow} and the first line of \eqref{eq:jump}, we get
$$
\|v_i-w_i\|_{C^{1,\alpha}(\overline{\Omega_{i,\lambda^k}\cap B_{3/4}})} \leq C \|w_2-w_1\|_{C^{1,\alpha}(\Gamma_{\lambda^k})}
= C \lambda^{-k \alpha} \|\psi_{\lambda^k}\|_{C^{1,\alpha}(\overline{B_1'})}.
$$
Using the normalization of $\psi$, we find that
$$\lambda^{-k\alpha}\|\psi_{\lambda^k}\|_{L^\infty(B_1')}
=  \sup_{x'\in B_1'} \frac{|\psi(\lambda^kx')|}{\lambda^{k(1+\alpha)}}
\leq [\psi]_{C^{1,\alpha}(0)} \leq \delta_0$$
$$\lambda^{-k\alpha}\|\nabla'\psi_{\lambda^k}\|_{L^\infty(B_1')} 
=  \sup_{x'\in B_1'} \frac{|\nabla'\psi(\lambda^kx')|}{\lambda^{k\alpha}} 
\leq [\psi]_{C^{1,\alpha}(0)} \leq \delta_0$$
and
$$\lambda^{-k\alpha}[\nabla'\psi_{\lambda^k}]_{C^{0,\alpha}(\overline{B_1'})}
=  \sup_{\substack{x', y'\in \overline{B_1'}\\x'\neq y'}} \frac{|\nabla'\psi(\lambda^kx')-\nabla'\psi(\lambda^ky')|}{\lambda^{k\alpha}|x'-y'|^\alpha} 
\leq [\psi]_{C^{1,\alpha}(\overline{B_1'})} \leq \delta_0.$$
In particular, it follows that
$$
\|(v_i-w_i)_{\nu_{\lambda^k}}\|_{L^\infty(\Gamma_{\lambda^k} \cap B_{3/4})} \leq  C\delta_0. 
$$
Moreover, for $x\in \Gamma_{\lambda^k}$,
\begin{align*}
| \tilde{g}(x)| & \leq  \frac{|g(\lambda^k x)-1|}{\lambda^{k\alpha}}+ \frac{|1-\nu_n (\lambda^k x)|}{\lambda^{k\alpha}} 
\leq [g]_{C^{0,\alpha}(0)} + [\nu_n]_{C^{0,\alpha}(0)} 
\leq \delta_0 + \delta_0 =2\delta_0.
\end{align*}
Hence, choosing $\delta_0$ sufficiently small, we see that 
\begin{align*}
\|\hat{g}\|_{L^\infty(\Gamma_{\lambda^k} \cap B_{3/4})} 
&\leq \|(v_1-w_1)_{\nu_{\lambda^k}}\|_{L^\infty(\Gamma_{\lambda^k} \cap B_{3/4})}
+ \|(v_2-w_2)_{\nu_{\lambda^k}}\|_{L^\infty(\Gamma_{\lambda^k} \cap B_{3/4})}
+ \|\tilde g\|_{L^\infty(\Gamma_{\lambda^k})}\\
&\leq C\delta_0 + C\delta_0 + 2\delta_0 = 2(C+1)\delta_0 \leq \delta.
\end{align*}
We have proved that ${v}\in C(\overline{B_1})$ satisfies 
$$\begin{cases}
\Delta {v} = {\hat g} \, dH^{n-1} \big |_{ \Gamma_{\lambda^k}} &\hbox{in}~B_1\\
|v|\leq 1&\hbox{in}~B_1\\
|{\hat g} | \leq \delta & \hbox{on}~\Gamma_{\lambda^k}\cap B_{3/4}.
\end{cases}$$
Therefore, we can apply Lemma~\ref{lem:harmonic} to $ v$
to find a linear polynomial ${P}(x) = {A}\cdot x + {B}$, with ${A}\in \R^n$, ${B}\in \R$
and $|{A}|+| B|\leq C_0$, such that
\begin{equation*}
|v(x)-{P}(x)|\leq \frac{\lambda^{1+\alpha}}{2}\qquad \qquad~\hbox{for all}~x\in B_{\lambda}.
\end{equation*}

Hence, for any $x\in B_\lambda$, by the estimate above and \eqref{eq:vclosetow},
\begin{align*}
|w(x)-{P}(x)| &\leq |w(x)-v(x)| + |v(x)-{P}(x)|
\leq \delta_0 + \frac{{\lambda}^{1+\alpha}}{2}
\leq \lambda^{1+\alpha}
\end{align*}
since $\delta_0 \leq \lambda^{1+\alpha}/2.$ According to \eqref{eq:rescale}, 
$$
\left | \frac{u(\lambda^k x) - \mathcal{P}_k(\lambda^k x)}{\lambda^{k(1+\alpha)}}-{P}(x)\right|\leq \lambda^{1+\alpha}\qquad \qquad~\hbox{for all}~x\in B_{\lambda}
$$
or equivalently, for $y=\lambda^k x$,
$$
|u(y) - \mathcal{P}_k(y)-\lambda^{k(1+\alpha)} {P}(y/\lambda^k)| \leq \lambda^{(k+1)(1+\alpha)}
\quad\qquad ~\hbox{for all}~y\in B_{\lambda^{k+1}}.
$$
Define the polynomials $P_{k+1}$ and $Q_{k+1}$ as
$$
P_{k+1}(y)=P_k(y)+\lambda^{k(1+\alpha)} {P}(y/\lambda^k), \qquad 
Q_{k+1}(y)=Q_k(y)+\lambda^{k(1+\alpha)} {P}(y/\lambda^k).
$$
From the previous estimate, it follows that
\begin{align*}
|u_1(y)-P_{k+1}(y)| &\leq \lambda^{(k+1)(1+\alpha)}\qquad \qquad~\hbox{for all}~y\in {\Omega_1}\cap B_{\lambda^{k+1}}\\
|u_2(y)-Q_{k+1}(y)| &\leq \lambda^{(k+1)(1+\alpha)}\qquad \qquad~\hbox{for all}~y\in {\Omega_2}\cap B_{\lambda^{k+1}}.
\end{align*}
Moreover, since $P_k(0)=Q_k(0)$, and $\nabla'P_{k}=\nabla'Q_{k}$, it is clear that $P_{k+1}(0)=Q_{k+1}(0)$, and  $\nabla'P_{k+1}=\nabla'Q_{k+1}$. Also, $(P_{k+1})_{x_n} -(Q_{k+1})_{x_n}=(P_{k})_{x_n} -(Q_{k})_{x_n}=1.$
If $P_{k+1}(y)= A_{k+1}\cdot y + B_{k+1}$ and $Q_{k+1}(y)= C_{k+1}\cdot y + B_{k+1}$ then 
\begin{align*}
A_{k+1} = A_k+ \lambda^{k\alpha}  A, \quad B_{k+1}= B_k + \lambda^{k(1+\alpha)} B, \quad C_{k+1}=C_k + \lambda^{k\alpha}  A.
\end{align*}
By the estimate $| A|+| B|\leq C_0$, we conclude
\begin{align*}
\lambda^k |A_{k+1}-A_k| + \lambda^k |C_{k+1}-C_k| +  |B_{k+1}-B_k| \leq C_0\lambda^{k(1+\alpha)}.
\end{align*}

The proof of the claim is completed.

\medskip

Finally, we consider the case $g(0)=0$. As before, it is enough to prove the following. \medskip

\noindent{\bf Claim.} {\it For all $k \geq 1$, there exists a linear polynomial $P_k=A_k \cdot x + B_k$  such that
\begin{align*}
\lambda^{k} |A_{k+1}-A_{k}| + |B_{k+1}-B_{k}|  \leq C_0 \lambda^{k(1+\alpha)}
\end{align*}
where $C_0=C_0(n,\alpha)>0$, and such that
\begin{align*}
|u(x)-P_k(x)| &\leq \lambda^{k(1+\alpha)}\qquad \qquad~\hbox{for all}~x\in {\Omega}\cap B_{\lambda^k}.
\end{align*}}

The proof is by induction. For $k=1$, since $\| u\|_{L^\infty(B_1)}\leq 1$, and
$$
\|g\|_{L^\infty(\Gamma)} =\sup_{x\in \Gamma}  |g(x)-g(0)|\leq \delta_0
$$
we can apply Lemma~\ref{lem:harmonic} to $u$. Then we find a linear polynomial $P_1(x) = A_1\cdot x +B _1$, with $A_1\in \R^n$, $B_1\in \R$, and $|A_1|+|B_1|\leq {C}_0$, such that
$$
|u(x)-P_1(x)|\leq \lambda^{1+\alpha} \qquad\hbox{for all}~x\in B_{\lambda}.
$$
Assume the claim holds for $k\geq 1$. Define
$$
w(x) = \frac{u(\lambda^k x)-P_k(\lambda^k x)}{\lambda^{k(1+\alpha)}} \qquad \hbox{for}~x\in \overline{B_1}.
$$ 
Then, for any $\varphi\in C^\infty_c(B_1)$,
$$
\Delta w(\varphi) =  \frac{\Delta (u(\lambda^k x))(\varphi)}{\lambda^{k(1+\alpha)}} =\int_{\Gamma_{\lambda^k}} \frac{g(\lambda^kx)}{\lambda^{k\alpha}} \, \varphi(x) \, dH^{n-1}.
$$
Also, for any $x\in \Gamma_{\lambda^k}$,
$$
\frac{|g(\lambda^kx)|}{\lambda^{k\alpha}} =\frac{|g(\lambda^kx)-g(0)|}{\lambda^{k\alpha}}  \leq [g]_{C^{0,\alpha}(0)} \leq \delta_0.
$$
Then the claim follows for $k+1$ by applying again Lemma~\ref{lem:harmonic}.
\qed

\section{Proof of Theorem \ref{thm:main}}\label{Section:proofofmain}

To prove Theorem~\ref{thm:main} we need Campanato's characterization of $C^{1,\alpha}$ spaces \cite{Campanato2}
and a technical result that patches the interior and boundary estimates together. We believe that the latter belongs
to the folklore (see, for example, \cite{Milakis-Silvestre}) but, for the sake of completeness, we will give a proof.

\begin{thm}[Campanato]\label{thm:Campanato}
Let $u$ be a measurable function defined on a bounded $C^{1,\alpha}$~domain $\Omega$.
Then $u\in C^{1,\alpha}(\overline{\Omega})$ if and only if there exists
$C_0>0$ such that for any $x\in\overline{\Omega}$, there exists a linear polynomial $Q_x(z)$ such that
$$|u(z)-Q_x(z)|\leq C_0|x-z|^{1+\alpha}$$
for all $z\in B_1(x)\cap\Omega$. In this case, if $C_\ast$ denotes the least constant $C_0>0$ for which the
property above holds, then 
$$
\|u\|_{C^{1,\alpha}(\overline{\Omega})}\sim C_\ast+\sup_{x\in\overline{\Omega}}|Q_x|,
$$
where $|Q_x|$ denotes the sum of the coefficients of the polynomial $Q_x(z)$.
\end{thm}

\begin{prop}\label{thm:regularity}
Let $S$ be a collection of measurable functions defined on a bounded $C^{1,\alpha}$~domain $\Omega$. For $x\in\Omega$, we let $d_x=\dist(x,\partial\Omega)$. Fix $u\in S$, and suppose the following hold.
\begin{enumerate}[$(i)$]

\item (Interior estimates).  There exist $A,C,D>0$ such that for any $x\in\Omega$ there exists a linear polynomial $P_x(z)$ such that
$$\|P_x\|_{L^\infty(B)}+d_x\|\nabla P_x\|_{L^\infty(B)}\leq C\|u\|_{L^\infty(B)}$$
and
$$
|u(z)-P_x(z)|\leq\bigg(A\frac{\|u\|_{L^\infty(B)}}{d_x^{1+\alpha}}+D\bigg)|z-x|^{1+\alpha}
$$
for all $z\in B\equiv B_{d_x/2}(x)$. \medskip

\item (Boundary estimates). There exists $E>0$ such that for any $y\in\partial\Omega$, there is a linear polynomial $P_y(z)$ such that
$$\|P_y\|_{L^\infty({\Omega})}+\|\nabla P_y\|_{L^\infty({\Omega})}\leq E$$
and
$$|u(z)-P_y(z)|\leq E|z-y|^{1+\alpha}$$
for all $z\in\overline{\Omega}$. \medskip

\item (Invariance property). For any $u\in S$, and any $y\in\partial\Omega$, with
corresponding linear polynomial $P_y$ as in $(ii)$, the function $v=u-P_y$ also satisfies the estimates of $(i)$.
\end{enumerate} 

Then $S\subset C^{1,\alpha}(\overline{\Omega})$, and there exists $M>0$, depending only on $A, C, D, E$ such that
$$
\|u\|_{C^{1,\alpha}(\overline\Omega)} \leq M \|u\|_{L^\infty(\Omega)}.
$$
\end{prop} 

\begin{proof}
We need to show that any $u\in S$ satisfies the Campanato characterization from Theorem \ref{thm:Campanato}.
Let us pick any point $x\in\overline{\Omega}$. If $x\in\partial\Omega$ then  the polynomial $Q_x(z)\equiv P_x(z)$, where $P_x(z)$ is as in assumption $(ii)$, satisfies the Campanato condition with $C_0=E$. 

Suppose next that $x\in\Omega$. Let $y\in\partial\Omega$ be a boundary point that realizes the distance from $x$ to the boundary, namely, $d_x=|x-y|$. Let $P_y(z)$ be the linear polynomial that satisfies $(ii)$. Consider the function $v(z)=u(z)-P_y(z)$.
By $(iii)$, there is a linear polynomial $P_x(z)$ such that the conditions in $(i)$ are met for
$v$ in place of $u$. We claim that the polynomial $Q_x$ for the Campanato condition is
$$Q_x(z)\equiv P_y(z)+P_x(z).$$
To show this, we split the argument into two cases. \medskip

\noindent\textbf{Case 1.} Suppose that $|z-x|<d_x/2$. This is the case when we can apply $(i)$ for $v-P_x$:
\begin{align*}
|u(z)-Q_x(z)| &= |u(z)-P_y(z)-P_x(z)| = |v(z)-P_x(z)| \\
&\leq \bigg(A\frac{\|v\|_{L^\infty(B_{d_x/2}(x))}}{d_x^{1+\alpha}}+D\bigg)|z-x|^{1+\alpha} \\
&= \bigg(A\frac{\|u-P_y\|_{L^\infty(B_{d_x/2}(x))}}{d_x^{1+\alpha}}+D\bigg)|z-x|^{1+\alpha}.
\end{align*}
Now, we notice that, by $(ii)$, by the choice of $y$, and the fact that $|z-x|<d_x/2$,
$$|u(z)-P_y(z)|\leq E|z-y|^{1+\alpha}\leq E(3/2d_x)^{1+\alpha}\leq 2^{1+\alpha} E d_x^{1+\alpha}.$$
Hence,
$$|u(z)-Q_x(z)|\leq (2^{1+\alpha} AE+D)|z-x|^{1+\alpha}$$
and $C_0=2^{1+\alpha}AE+D$. \medskip

\noindent\textbf{Case 2.} Suppose that $|z-x|\geq d_x/2$. By the estimate in $(i)$ for $P_x(z)$, we get
\begin{align*}
|P_x(z)| 
&= |P_x(x) + \nabla P_x(x)\cdot(z-x)|\\
&\leq C\|u-P_y\|_{L^\infty(B)}+ C d_x^{-1} \|u-P_y\|_{L^\infty(B)}|z-x|.
\end{align*}
Also, by the boundary estimate in $(ii)$,
\begin{align*}
\|u-P_y\|_{L^\infty(B)} &\leq ({3}/{2})^{1+\alpha} E d_x^{1+\alpha}.
\end{align*}
Hence,
\begin{align*}
|u(z)-Q_x(z)| &\leq |u(z)-P_y(z)|+|P_x(z)| \\
&\leq E|z-y|^{1+\alpha} +C\|u-P_y\|_{L^\infty(B)}+ C d_x^{-1} \|u-P_y\|_{L^\infty(B)}|z-x|\\
&\leq 3^{1+\alpha}E |z-x|^{1+\alpha} +3^{1+\alpha} C E d_x^{1+\alpha}+C d_x^{-1} ({3}/{2})^{1+\alpha} E d_x^{1+\alpha}|z-x| \\
&\leq 3^{1+\alpha}E(1+2C)|z-x|^{1+\alpha}.
\end{align*}
Thus, in this case, the Campanato constant is $C_0=3^{1+\alpha}E(1+2C) $.
\end{proof}

\begin{proof}[Proof of Theorem~\ref{thm:main}]
Let $u\in {\rm LogLip}(\overline\Omega)$ be the solution given by Theorem~\ref{thm:existence}. We will show the statement for the function $u_2: \overline{\Omega}_2\to \R$,  and we can argue similarly for $u_1:\overline{\Omega}_1\to \R$. The following holds.
\begin{enumerate}[$(i)$]
\item {\it (Interior estimates).} For any $x\in \Omega_2$, there exists a linear polynomial $P_x(z)$ such that
$$
\|P_x\|_{L^\infty(B)}+d_x\|\nabla P_x\|_{L^\infty(B)}\leq (1+2n)\|u_2\|_{L^\infty(B)}
$$
and
$$
|u_2(z)-P_x(z)|\leq 2^{\alpha-1} n \frac{\|u\|_{L^\infty(B)}}{d_x^{1+\alpha}} |z-x|^{1+\alpha} 
$$
for all $z\in B\equiv B_{d_x/2}(x)$. 

Indeed, fix $x\in \Omega_2$. Since $u_2$ is harmonic, it is smooth in $\Omega_2$, so we can define
$$
P_x(z) = u_2(x)+\nabla u_2(x) \cdot (z-x).
$$
Then, by classical interior estimates for harmonic functions, 
\begin{align*}
\|P_x\|_{L^\infty(B)}+d_x\|\nabla P_x\|_{L^\infty(B)}  &\leq \|u_2\|_{L^\infty(B)} + d_x \|\nabla u_2\|_{L^\infty(B)} +d_x\|\nabla u_2\|_{L^\infty(B)} \\
& \leq \|u_2\|_{L^\infty(B)} + 2n  \|u_2\|_{L^\infty(B)}\\
& \leq (1+2n) \|\nabla u_2\|_{L^\infty(B)}.
\end{align*}
Moreover, 
\begin{align*}
|u_2(z)-P_x(z)| & \leq \|D^2 u_2\|_{L^\infty(B)}|z-x|^2\\
&\leq n \frac{\| u_2\|_{L^\infty(B)}}{d_x^2}  |z-x|^2\leq 2^{\alpha-1} n \frac{\| u_2\|_{L^\infty(B)}}{d_x^{1+\alpha}}  |z-x|^{1+\alpha}.
\end{align*}

\item {\it (Boundary estimates).} Consider $\partial\Omega_2=  \Gamma \cup \partial \Omega $.

If $y\in \Gamma$, by Theorem~\ref{thm:boundary}, there exists a linear polynomial $P_y(z)$ such that
$$
\|P_y\|_{L^\infty(\Omega_2)} + \|\nabla P_y\|_{L^\infty(\Omega_2)} \leq E
$$
and
\begin{equation*}
|u_2(z)-P_y(z)|\leq E |z-y|^{1+\alpha} 
\end{equation*}
for all $z\in \overline{\Omega}_2$, with
$
E \leq C_0 \|\psi\|_{C^{1,\alpha}(B_1')} \|g\|_{C^{0,\alpha}(\Gamma)},
$
and $C_0=C_0(n,\alpha)>0$. 

If $y\in \partial\Omega\in C^\infty$, since $u_2=0$, then, by classical boundary regularity for harmonic functions, $u_2\in C^{1,\alpha}(\overline{B\cap \Omega})$, with $B\equiv B_r(y)$, for some $r>0$ sufficiently small. By  Theorem~\ref{thm:Campanato}, there exists a linear polynomial $P_y(z)$ such that
$$
|u_2(z)-P_y(z)|\leq C_0 |z-y|^{1+\alpha}
$$
for all $z\in \overline{\Omega}_2$, for some $C_0(n,\alpha)>0$.\medskip

\item {\it (Invariance property).} Fix $y\in \partial \Omega_2$, and let $P_y$ be the corresponding linear polynomial given in $(ii)$. Clearly, the function $v= u_2-P_y$ is harmonic in $\Omega_2$, so it satisfies the interior estimates in $(i)$.
\end{enumerate}

Therefore, by Theorem~\ref{thm:regularity}, we have $u_2\in C^{1,\alpha}(\overline{\Omega}_2)$, and there exists a constant $C>0$, depending only on
$n$, $\alpha$ and $\Gamma$ such that $\|u_2\|_{C^{1,\alpha}(\overline{\Omega}_2)}\leq C \|g\|_{C^{0,\alpha}(\Gamma)}$.
\end{proof}

\section{Appendix}\label{app:special}

A special Lipschitz domain $\Omega$ in $\R^n$ is a set of the form
$$\Omega=\{(x',x_n)\in\R^n:x'\in\R^{n-1},\,x_n>\psi(x')\}$$
where $\psi\in\operatorname{Lip}(\R^{n-1})$, that is, there exists $M>0$ such that
$$
|\psi(x')-\psi(y')|\leq M|x'-y'|\qquad\hbox{for all}~x',y'\in\R^{n-1}.
$$
In other words, $\Omega$ is the set of points lying above the graph of a Lipschitz function $\psi$.
Then, by Rademacher's Theorem, $\psi$ is Fr\'echet differentiable almost everywhere with
$\|\nabla\psi\|_{L^\infty(\R^{n-1})}\leq M$. On $\partial\Omega$ we thus have
$$dH^{n-1}\big|_{\partial\Omega}=\sqrt{1+|\nabla\psi(x')|^2}\,dx'\quad\hbox{and}\quad
\nu(x',\psi(x'))=\frac{(\nabla\psi(x'),-1)}{\sqrt{1+|\nabla\psi(x')|^2}}$$
where $x=(x',\psi(x'))\in\partial\Omega$. For a measurable function $f$ on $\partial\Omega$, we have
$$
\int_{\partial\Omega}f(x)\,dH^{n-1}=\int_{\R^{n-1}}f(x',\psi(x'))\sqrt{1+|\nabla\psi(x')|^2}\,dx'.
$$
For more details see \cite{Evans-Gariepi,Maggi}.

A bounded Lipschitz domain in $\R^n$ is a bounded domain $\Omega$ such that the boundary $\partial\Omega$ can be covered by finitely many open balls $B_j$ in $\R^n$, $j=1,\ldots,J$, centered at $\partial\Omega$, such that
$$B_j\cap\Omega=B_j\cap\Omega_j,\quad j=1,\ldots,J$$
where $\Omega_j$ are rotations of suitable special Lipschitz domains
given by Lipschitz functions $\psi_j$. 
One may then assume that $\partial\Omega\cap B_j$ can be represented in local coordinates by
$x_n=\psi_j(x')$, where $\psi_j$ is a Lipschitz function on $\R^{n-1}$ with $\psi_j(0')=0$. Recall also that if $\psi$ is a Lipschitz function defined on an set $A\subset\R^{n-1}$, with Lipschitz constant $M$, then there exists an extension $\overline{\psi}:\R^{n-1}\to\R$ of $\psi$ such that $\overline{\psi}=\psi$ on $A$ and the Lipschitz constant of $\overline{\psi}$ does not exceed $M$, see \cite{Evans-Gariepi}.  

Let $\Omega_0=\Omega\cap\big(\bigcup_{j=1}^JB_j\big)^c$.
A partition of unity $\{\xi_j\}_{j=0}^J$ subordinated to $\{\Omega_0,B_1,\ldots,B_J\}$
is a family of nonnegative smooth functions $\xi_j$ on $\R^n$ such that
$$\xi_0\in C^\infty_c(\Omega_0)\quad\xi_j\in C^\infty_c(B_j),~j=1,\ldots,J\quad\hbox{and}\quad\sum_{j=0}^J\xi_j(x)=1\quad\hbox{for all}~x\in\overline{\Omega}.$$
It follows that $0\leq\xi_j\leq1$, $j=0,1,\ldots,J$.
Obviously the family $\{\xi_j\}_{j=1}^J$ is a partition of unity subordinated to 
the open cover $\{B_1,\ldots,B_J\}$ of $\partial\Omega$ and 
$\sum_{j=1}^J\xi_j(x)=1$ for every $x\in\partial\Omega$. 

Let $f:\Gamma\to\R$ be a measurable function, where $\Gamma=\partial\Omega$ is the boundary of a bounded
Lipschitz domain $\Omega$. Consider the balls $B_j$, $j=1,\dots,J$, that cover $\Gamma$ as above, and the corresponding Lipschitz functions $\psi_j:\R^{n-1}\to\R$. Let $\{\xi_j\}_{j=1}^J$ be a smooth partition of unity subordinated to the open
cover $\{B_j\}_{j=1}^J$ of $\Gamma$. Then
$$\int_\Gamma f\,dH^{n-1}=\sum_{j=1}^J\int_\Gamma \xi_jf\,dH^{n-1}=\sum_{j=1}^J\int_{B_j\cap\Gamma}\xi_jf\,dH^{n-1}.$$
Let us consider each one of the terms in the sum above separately.
We study the following situation: let $B$ be a ball and let $\bar{f}:B\cap\Gamma\to\R$ of compact support in $B\cap\Gamma$.
Let $\psi:\R^{n-1}\to\R$ be a Lipschitz function such that
$\psi(B_1')=B\cap\Gamma$.
Then, by extending trivially $\bar{f}$ to the rest of the graph of $\psi$ and using the coarea formula \cite{Evans-Gariepi,Maggi},
\begin{align*}
\int_{B\cap\Gamma}\bar{f}\,dH^{n-1} &=\int_{\psi(B_1')}\bar{f}\,dH^{n-1}=\int_{\psi(\R^{n-1})}\bar{f}\,dH^{n-1} \\
&=\int_{\R^{n-1}}\bar{f}(y',\psi(y'))\sqrt{1+|\nabla\psi(y')|^2}\,dy' \\
&=\int_{B_1'}\bar{f}(y',\psi(y'))\sqrt{1+|\nabla\psi(y')|^2}\,dy'.
\end{align*}

\smallskip

\noindent\textbf{Remark.}~The contents of this work are part of the second author's PhD dissertation.
She presented these results at the AMS Fall Central Sectional Meeting at University of Michigan, Ann Arbor (Oct.~2018),
the Barcelona Analysis Conference at Universitat de Barcelona (Jun.~2019), and the Midwest Geometry Conference 
at Iowa State University (Sep.~2019).

\medskip

\noindent\textbf{Acknowledgments.} We would like to thank the referees for some very valuable remarks that helped us
improve the presentation of this paper.




\begin{thebibliography}{10}

\bibitem{Blank-Hao} I.  Blank and Z. Hao,
The mean value theorem and basic properties of the obstacle problem for divergence form elliptic operators,
\textit{Comm. Anal. Geom.}
\textbf{23} (2015), 129--158. 

\bibitem{Caffarelli} L. A. Caffarelli, 
Elliptic second order equations,
\textit{Rend. Sem. Mat. Fis. Milano}
\textbf{58} (1988), 253--284.

\bibitem{Campanato} S. Campanato, 
Sul problema di M. Picone relativo all'equilibrio di un corpo elastico incastrato,
\textit{Ricerche Mat.}
\textbf{6} (1957), 125--149.

\bibitem{Campanato2} S. Campanato, 
Propriet\`{a} di h\"{o}lderianit\`{a} di alcune classi di funzioni, 
\textit{Ann. Scuola Norm. Sup. Pisa Cl. Sci. (3)}
\textbf{17} (1963), 175--188.

\bibitem{Citti-Ferrari} G.~Citti and F.~Ferrari,
{A sharp regularity result of solutions of a transmission problem},
\textit{Proc. Amer. Math. Soc.}
\textbf{140} (2012), 615--620.

\bibitem{Evans-Gariepi} L. C. Evans and R. F. Gariepi,
\textit{Measure Theory and Fine Properties of Functions},
Studies in Advanced Mathematics,
CRC Press, Boca Raton, FL, 1992.

\bibitem{Gilbarg-Trudinger} D. Gilbarg and N. S. Trudinger, 
\textit{Elliptic partial differential equations of second order},
Classics in Mathematics, Springer-Verlag, Berlin, 2001.

\bibitem{Ladyzhenskaya-Uraltseva} O. A. {Ladyzhenskaya} and N. N. {Ural'tseva},
\textit{Linear and Quasilinear Elliptic Equations},
Academic Press, New York-London, 1968.

\bibitem{Li-Nirenberg} Y.~Li and L.~Nirenberg,
{Estimates for elliptic systems from composite material. Dedicated to the memory of J\"urgen K. Moser},
\textit{Comm. Pure Appl. Math.}
\textbf{56} (2003), 892--925.

\bibitem{Li-Vogelius} Y.~Li and M.~Vogelius,
{Gradient estimates for solutions to divergence form elliptic equations with discontinuous coefficients},
\textit{Arch. Ration. Mech. Anal.}
\textbf{153} (2000), 91--151.

\bibitem{Littman-Stampacchia-Weinberger} W. Littman, G. Stampacchia and H. F. Weinberger, 
Regular points for elliptic equations with discontinuous coefficients, 
\textit{Ann. Scuola Norm. Sup. Pisa Cl. Sci. (3)}
\textbf{17} (1963), 43--77.

\bibitem{Lions}  J. L. Lions,
Contributions \`a un probl\`eme de M. M. Picone, 
\textit{Ann. Mat. Pura Appl. (4)}
\textbf{41} (1956), 201--219.

\bibitem{Maggi} F. Maggi,
\textit{Sets of Finite Perimeter and Geometric Variational Problems. An Introduction to Geometric Measure Theory},
Cambridge Studies in Advanced Mathematics \textbf{135},
Cambridge University Press, Cambridge, 2012.

\bibitem{Mateu-Orobitg-Verdera} J.~Mateu, J.~Orobitg and J.~Verdera,
{Extra cancellation of even Calder\'on-Zygmund operators and quasiconformal mappings},
\textit{J. Math. Pures Appl. (9)}
\textbf{91} (2009), 402--431.

\bibitem{Milakis-Silvestre} E. Milakis and L. E. Silvestre,
Regularity for fully nonlinear elliptic equations with Neumann boundary data,
\textit{Comm. Partial Differential Equations}
\textbf{31} (2006), 1227--1252.

\bibitem{Picone} M.~Picone, 
Sur un probl\`eme nouveau pour l'\'equation lin\'{e}aire aux d\'eriv\'{e}es partielles de la th\'{e}orie mathematique classique de l'\'elasticit\'e, 
\textit{Colloque sur les \'{e}quations aux d\'{e}riv\'{e}es partielles}, Bruxelles, May 1954.
 
\bibitem{Schechter} M. Schechter,
A generalization of the problem of transmission,
\textit{Ann. Scuola Norm. Sup. Pisa Cl. Sci. (3)}
\textbf{14} (1960), 207--236.
 
\bibitem{Stampacchia} G. Stampacchia
Su un problema relativo alle equazioni di tipo ellittico del secondo ordine, 
\textit{Ricerche Mat.}
\textbf{5} (1956), 3--24.

\end{thebibliography}
\end{document}